\documentclass{amsart}
\usepackage{enumitem}
\usepackage[margin=2.5cm]{geometry}
\usepackage[utf8]{inputenc}
\usepackage{epigraph}
\usepackage[all]{xy}
\usepackage{url}
\usepackage{hyperref}
\hypersetup{
colorlinks,
linkcolor={red!50!black},
citecolor={blue!50!black},
urlcolor={blue!80!black}
}
\usepackage{amsfonts}
\usepackage{amsmath}
\usepackage{amssymb}
\usepackage{extpfeil}
\usepackage{amsthm}
\usepackage{mathtools}
\usepackage{tikz-cd}
\usepackage{xcolor}
\usepackage{caption}

\newcommand{\var}{\textup{Var}}

\newcommand{\C}{\mathcal{C}}

\theoremstyle{definition}
\newtheorem{defn}{Definition}[section]
\newtheorem{ex}[defn]{Example}

\newtheorem{rmk}[defn]{Remark}

\newtheorem{cor}[defn]{Corollary}
\theoremstyle{plain}

\newtheorem{lem}[defn]{Lemma}
\newtheorem{prop}[defn]{Proposition}

\newtheorem{thmx}{Theorem}
\newtheorem{propx}[thmx]{Proposition}

\newcommand{\defin}{\mathrm{Def}}

\newcommand{\im}{\textup{Im}}
\newcommand{\D}{\mathcal{D}}
\newcommand{\op}{\textup{op}}
\newcommand{\lc}{\textup{lc}}

\renewcommand{\hom}{\textup{Hom}}
\newcommand{\set}{\textup{Set}}
\newcommand{\M}{\mathcal{M}}
\newcommand{\E}{\mathcal{E}}
\newcommand{\A}{\mathcal{A}}
\newcommand{\W}{\mathcal{W}}

\newcommand{\R}{\mathbb{R}}
\renewcommand{\P}{\mathcal{P}}
\renewcommand{\min}{\textup{min}}
\newcommand{\ch}{\textup{Ch}}
\newcommand{\hb}{\textup{hb}}
\newcommand{\Z}{\mathbb{Z}}
\newcommand{\supp}{\textup{supp}}
\newcommand{\hfun}{\textup{hFun}}

\title{All $K$-theory is squares $K$-theory}
\author{Josefien Kuijper}
\vspace{-1cm}

\begin{document}

\begin{abstract}
   We show that the $K$-theory spectra of many assemblers, such as the assembler of polytopes in euclidean, hyperbolic or spherical geometry, as well as the assembler of definable sets, are equivalent to the $K$-theory spectrum of a squares category. We use this to lift the definable Euler characteristic of definable sets in an o-minimal structure to a map of $K$-theory spectra.
\end{abstract}
\maketitle

\vspace{-1cm}

\section*{Introduction}
\epigraph{\textit{Now I'm playing it real straight\\
And yes, I cut my hair\\
You might think I'm crazy\\
But I don't even care\\
'Cause I can tell what's going on}}{Huey Lewis and the News}

\subsection{Motivation}
There is a well-known analogy between the classical scissors congruence groups, the Grothendieck ring of varieties, and the Grothendieck group of a ring, or more generally, $K_0(\W)$ for a Waldhausen category $\W$. In all cases, the group can be generated by isomorphism classes of objects, modulo the relations that $[O]=0$ for some distinguished object $O$ (the zero module in algebraic settings, or the empty polytope or variety in geometric settings) and the three-term relation 
$$[A] + [C] = [B]$$
whenever $B$ decomposes as $A$ and $C$ in some way. This can mean: when there exists an ``exact sequence'' 
\begin{equation}\label{eq:exact_seq}
    A \to B \to C 
\end{equation}
expressing $C$ as a \textit{quotient} $B/A$ in the algebraic setting; or when there exists arrows
\begin{equation}\label{eq:subtraction_seq}
    A\to B\leftarrow C
\end{equation}
where $C = B\setminus A$ is the \textit{complement} of an open immersion of varieties $A \hookrightarrow B$, for the Grothendieck ring of varieties; or when $A\to B $ and $C\to B$ cover $C$ with an overlap of measure zero, in the case of scissors congruence groups of polytopes.

The formal study of scissors congruence groups as $K$-theory began with the theory of assemblers, developed by Zakharevich \cite{Z-Kth-ass}. Both the Grothendieck ring of varieties, and the classical scissors congruence groups, are the zeroth homotopy group of spectra $K(\var)$ and $K(\P^X_G)$ respectively. The input for such a spectrum is an \textit{assembler}, which is a category with the additional structure of ``finite disjoint covering families'', which encode how an object can be broken up into disjoint pieces. Assembler $K$-theory behaves similar to Waldhausen $K$-theory: there is a Dévissage Theorem and a Localization Theorem. However, in order to directly compare assembler and Waldhausen $K$-theory, one has to overcome a fundamental difference between assemblers and Waldhausen categories: in an assembler, all arrows point from smaller objects to bigger ones, cf. (\ref{eq:subtraction_seq}), whereas in a Waldhausen category, the bigger objects sits in the middle of a sequence of composable morphisms, cf. (\ref{eq:exact_seq}).

The recent framework of \textit{squares categories} \cite{squares}, among other things, aims to clear up this apparent contradiction. Central is the observation that a three-term relation can be replaced by a four-term relation.
A squares category $\C$ is, roughly, determined by a collection of distinguished squares
\begin{equation}\label{eq:intro_dist_square}
    \begin{tikzcd}
        A \arrow[r] \arrow[d] & B \arrow[d] \\ C \arrow[r] & D
    \end{tikzcd}
\end{equation}
which correspond to the relation $[A] + [D] = [B] + [C]$ in the zeroth $K$-group. The $K$-theory spectrum of a squares category is defined via a $T_\bullet$-construction which mimics the ``Thomason'' $wT_\bullet$-construction of a Waldhausen category in \cite[Section 1.3]{waldhausen}. It is immediate that any Waldhausen category can be reinterpreted as a squares category where the distinguished squares are homotopy pushout squares, and Waldhausen's proof shows that its $K$-theory as a squares category is the same as its $K$-theory as a Waldhausen category.

In geometric settings, it is natural to define a squares category where a square of the form (\ref{eq:intro_dist_square}) is distinguished if inclusions $B\to D$ and $C\to D$ cover $D$ with intersection $A$.
This gives rise to a squares category of varieties, whose $K$-theory can be shown to be equivalent to the $K$-theory the assembler of varieties \cite[Proposition 4.6]{squares}. Similarly, squares categories of polytopes and of definable sets can be defined. Moreover, there are examples of squares categories that do not fit the assembler framework, such as the squares category of (equivariant) manifolds with boundary \cite{manifolds}, \cite{equiv_mfolds}. Its squares $K$-theory spectrum lifts the ``SK-group with boundary'' $\textup{SK}^\partial_n$.

At this point, it is starting to look like all $K$-theory is indeed squares $K$-theory, but there is one caveat. For polytopes and definable sets, unlike varieties, it has been unclear if the natural squares categories that one can define in these contexts, have the same $K$-theory spectra as the assemblers of polytopes and definable sets respectively; it has only been shown that the $K_0$-groups agree. There has been an interest especially in the assembler $K$-theory spectra of polytopes in euclidean, spherical and hyperbolic geometry, and their (higher) homotopy groups (\cite{cz-hilbert},\cite{scissors_thom},\cite{trace_map}, \cite{KKMMW1}), and Barton and Zakharevich have announced results on $K_1$ of the assembler of definable sets in an o-minimal structure \cite{talk_inna}. If these spectra can be shown to be equivalent to the $K$-theory spectrum of a squares category, then this provides an new possible approach to this subject. Moreover, this can make it easier to construct maps to and from $K$-theory spectra of Waldhausen categories, as is illustrated by Proposition \ref{propx:intro_appl}.

\subsection{Outline}

Our main result is the following.
\begin{thmx}[{Proposition \ref{prop:artificial_squares_cat_works}}]\label{thmx:comparison}
    Let $\A$ be an assembler satisfying certain axioms. Then there exists a squares category $\C^\min$ such that $K^\square(\C^\min) \simeq K(\A)$. 
\end{thmx}
The axioms in question are few, and they are fulfilled in the geometric examples that we care about: polytopes in an $n$-dimensional geometry $X$, and definable sets.

The process of constructing an equivalence between the $K$-theory spectrum of a squares category and the $K$-theory spectrum of an assembler, consists of two steps. By definition, the $K$-theory spectrum of a squares category is defined via the $T_\bullet$-construction. An ``$S^\square_\bullet$-construction'', analogous to the $S_\bullet$-construction of a Waldhausen category, can also be defined, as is done in \cite[Definition 2.20]{Calle_Sarazola}. In Section \ref{sect:on_s_bullet}, we define a class of squares categories, called \textit{squares categories with complements}.
We then show the following.
\begin{propx}[Proposition \ref{prop:Tplus_vs_T}]\label{propx:T_vs_Ssquare}
Let $\C$ be a squares category with complements, then there is an equivalence 
$$|T_\bullet \C | \simeq |S^\square_\bullet\C|$$
\end{propx}

Secondly, we compare the $S^\square_\bullet$-construction to the construction of the $K$-theory spectrum of an assembler. This is essentially what is done in \cite[Proposition 9.13]{devissage}, where the $K$-theory of the SW-category of varieties, obtained from Campbell's $\widetilde S_\bullet$-construction, is shown to be equivalent to the $K$-theory spectrum of the assembler of varieties. In Section \ref{sect:s_bullet_vs_assembler} we give an alternative argument that is very general. More precisely, we show the following.
\begin{propx}[Proposition \ref{prop:assmebler_K_th_is_squares_K_th}]\label{propx:K_square_vs_assembler}
Let $\C$ be a squares category with complements and $\A$ a category with covering families, that share an ambient category $\D$. Assume a number of conditions. Then there is an equivalence 
$$K^\square(\C) \simeq K(\A).$$
\end{propx}
We prove Theorem \ref{thmx:comparison} in Section \ref{subsect:assemblers} by constructing, given an assembler $\A$, a ``minimal'' squares category with complements $\C^\min$, and showing that the conditions of Proposition \ref{propx:K_square_vs_assembler} are fulfilled.
 
In the rest of Section \ref{sect:examples} we treat a few geometric examples. The minimal squares category that is constructed in the proof of Theorem \ref{thmx:comparison}, is rather artificial. In particular, it is different from the squares category of varieties, polytopes, or definable sets that one would naturally consider. 
However, in the cases of interest, we can construct other, more natural looking squares categories with complements, that also satisfy the conditions of Proposition \ref{propx:K_square_vs_assembler}. Moreover, we can compare these squares categories to ones appearing in the literature. This gives the following results.
\begin{propx}[{Corollary  \ref{cor:K_square_def_vs_K_pCGW_def} and Proposition \ref{prop:squares_cat_CS_agrees}}]
    \begin{itemize}
        \item[(1)]  The $K$-theory spectrum $K(\defin(R))$ of the assembler of definable sets is equivalent to the $K$-theory spectrum of the pCWG-category of definable sets defined in \cite[Example 1.13]{Ming}.
        \item[(2)]  The $K$-theory spectrum $K(\P^G_{E_n})$ of the assembler of euclidean $n$-dimensional polytopes is equivalent to the $K$-theory spectrum of the squares category $\mathbb{P}^n_G$ defined in \cite[Example 2.14]{Calle_Sarazola}.   \end{itemize}
\end{propx}

\subsection{Application}
Finally, in Section \ref{sect:application}, we consider the squares category of definable sets in an o-minimal structure. For $X$ a definable set in an o-minimal structure on $(R,<)$, one can write $X$ as a finite disjoint union of cells $C_i$, each definably homeomorphic to $R^n$ for some $n \geq 0$. This $n$ is called the dimension of $C_i$, and $\dim(X)$ is defined to be the maximum of the dimensions of the cells in any cell decomposition of $X$. Given a cell-decomposition $X= \bigcup_I C_i$, the \textit{definable Euler characteristic} of $X$ is defined as 
$$\chi(X) = \sum_{i\in I} (-1)^{\dim(C_i)}.$$
This is independent of the choice of cell-decomposition. For $D = B \cup C$, it is clear that $\chi(D) = \chi(B)+\chi(C) - \chi(B\cap C)$. Moreover, if there exist a definable bijection $X\to Y$, then we have $\chi(X) = \chi(Y)$. We refer to \cite{Dries} for more details.

The basic properties of $\chi$ imply that there is a well-defined map 
$$\chi:K_0(\defin(R)) \to \Z.$$ A natural question is then whether can we lift this map to a map 
$$K(\defin(R)) \to \mathcal{Z},$$
where $\mathcal{Z}$ is some spectrum with $\pi_0(\mathcal{Z}) = \Z$, possibly also a $K$-theory spectrum. With the description of $K(\defin(R))$ as the $K$-theory spectrum of a squares category, this turns out to be relatively straightforward.

In \cite[Section 5]{manifolds} a map of $K$-theory spectra $K^\square(\textup{Mfd}_n^\partial) \to K(\Z)$ out of the $K$-theory of a squares category of manifolds with boundary is defined, lifting the Euler characteristic of manifolds with boundary. We can apply the same arguments as those used in loc.cit, to show the following.
\begin{propx}[{Corollary \ref{cor:map_of_spectra_def_ch} and Proposition \ref{prop:map_lifts_eulerchar}}]\label{propx:intro_appl}
    There is a map of $K$-theory spectra 
    $$K(\A_{\defin(R)}) \to K(\Z/2)$$ which, on $\pi_0$, agrees with the definable Euler characteristic
    $$ \chi^\defin:\textup{Obj}(\defin(R))\to \Z.$$
\end{propx}

\subsection{Acknowledgements}
I thank Robin Sroka for insightful conversations that eventually led to this work. This work was partially supported by a grant from the Knut and Alice Wallenberg foundation.

\section{Preliminaries}

We recall the definition of a squares category and its $K$-theory.
\begin{defn}[{\cite[Definition 1.4]{squares}}]
    A \textit{squares category} $\C = (\E,\M)$ consists of categories $\E$ (also called the vertical morphisms, denoted $\twoheadrightarrow$) and $\M$ (also called the horizontal morphisms, denoted $\rightarrowtail$) which have the same objects, and a choice of base point object $O$, together with a collection of diagrams
    \begin{equation}\label{eq:dist_square}
         \begin{tikzcd}
            A \arrow[r, "f", tail] \arrow[d, "h", two heads] & B \arrow[d, "j", two heads] \\
            C \arrow[r, "g", tail] & C
        \end{tikzcd}  
    \end{equation}
   with $f,g$ in $\M$ and $h, j$ in $\E$,
    called \textit{distinguished squares},
    such that the following hold.
    \begin{itemize}
        \item The distinguished squares are closed under horizontal and vertical composition.
        \item For every $f:A\twoheadrightarrow B$ in $\E$ and every $g:C\rightarrowtail D$ in $\M$, the squares
        \begin{center}
            \begin{tikzcd}
                A \arrow[d, "f", two heads] \arrow[r, "="] & A \arrow[d, "f", two heads]\\
                B\arrow[r, "="]& B
            \end{tikzcd}
            and \begin{tikzcd}
                C \arrow[r, "g", tail]\arrow[d, "="] & D \arrow[d, "="]\\
                C \arrow[r,  "g", tail] & \D
            \end{tikzcd}
        \end{center}
        are distinguished.
        \item The object $O$ is initial in both $\E$ and $\M$. 
    \end{itemize}
\end{defn}
We note that a squares category can be seen as a simple double category, i.e., a double category where the 2-cells are uniquely determined by their boundary. 

Squares categories generalize Waldhausen categories as follows.
\begin{ex}[{\cite[Example 1.8]{squares}}] \label{ex:wald_square}
    Given a Waldhausen category $\W$, we define a squares category $\C_\W = (\E_\W,\M_\W)$ with $\E_\W = \W^\op$ and $\M_\W$ the category of cofibrations in $\W$, with $0$ the distinguished object. A square 
    \begin{center}
        \begin{tikzcd}
            A \arrow[r, hook] & B \\
            C \arrow[u] \arrow[r, hook]& D  \arrow[u]
        \end{tikzcd}
    \end{center}
    is distinguished if it commutes as a square in $\W$, and the induced map $A \cup_C D \to B$ is a weak equivalence.
\end{ex}

In many cases, $\E$ and $\M$ are wide subcategories of a category $\D$, and the distinguished squares are commutative squares in $\D$. In that case we say that $\D$ is the \textit{ambient category} of the squares category $\C$. 
\begin{defn}[{\cite[Definition 2.1]{squares}}]
  For $\C=(\E,\M)$ a squares category and $\D$ an ordinary category, we denote by $\hfun(\D,\C)$ the category whose objects are functors $F:\D\to \M$, and whose morphisms $F\to G$ are given by a choice of vertical morphism $\alpha_X:F(X)\twoheadrightarrow G(X)$ for every $X$ in $\D$, such that for every $f:X\to Y$ in $\D$ the square
\begin{center}
    \begin{tikzcd}
        F(X) \arrow[r, "F(f)", tail] \arrow[d, "\alpha_X",two heads] & F(Y) \arrow[d, "\alpha_Y", two heads] \\
        G(X) \arrow[r, "G(f)",tail] & G(Y)
    \end{tikzcd}
\end{center}
is distinguished.  
\end{defn}
We denote by $T_n\C$ the category $\hfun([n],\C)$. Then the categories $T_n\C$ assemble into a simplicial category $T_\bullet \C$, whose face and degeneracy maps are given by composing and inserting identities respectively, as in the nerve construction. 
\begin{defn}[{\cite[Definition 2.2]{squares}}]
    Given a squares category $\C=(\E, \M)$, we define its $K$-theory space
    $$K^\square(\C) = \Omega_O|N_\bullet T_\bullet \C|.$$
\end{defn}
For $\W$ a Waldhausen category, $T_n \C_\W$ is the opposite category of $wT_n \W$ as defined in \cite[Section 1.3]{waldhausen}, and arguments in loc. cit. imply that $K^\square(\C_\W)$ is equivalent the the Waldhausen $K$-theory space as defined via the $S_\bullet$-construction.\\

A squares category $\C = (\E,\M)$ is called a \textit{symmetric monoidal squares category} \cite[Definition 1.6]{squares} if there are symmetric monoidal structures on $\E$ and $\M$ which agree on objects and for which $O$ is the unit, such that the product of distinguished squares is a distinguished square. If $\C$ is a symmetric monoidal squares, category, then  $K^\square(\C)$ is an infinite loop space by \cite[Theorem 2.5]{squares}. In practice, the squares categories that we consider will all be symmetric monoidal, with the symmetric monoidal structure given by the coproduct or restricted pushout over $O$.

\begin{defn}\label{defn:rest_coprod}
   For $\C$ a category with an initial object $O$ and objects $A,B$, the \textit{restricted pushout over $O$} (also called the \textit{restricted coproduct}) is the initial object $X$ which fits into a pullback diagram
    \begin{center}
        \begin{tikzcd}
            O \arrow[r] \arrow[d] & A \arrow[d] \\
            B \arrow[r] & X.
        \end{tikzcd}
    \end{center}
    As an example, in the category of finite sets and injections, the restricted pushout of $A$ with $A$ over $\emptyset$ is the disjoint union $A\amalg A$, but this is not the coproduct, since the fold map $A\amalg A \to A$ is not an injection. See also e.g. \cite[Definition 5.3]{devissage}.
\end{defn}
The following proposition shows that $K^\square_0(\C)$ usually has an easy characterization. This also implies that for many assemblers, including the assemblers of polytopes and definable sets, it is straightforward to define a squares category whose $K_0$-group agrees with the assembler $K_0$-group. \begin{prop}[{\cite[Theorem 3.1]{squares}}]
    Let $\C$ be a squares category with the additional property that for any two objects $A,B$, there is a third object $X$ such that there are distinguished squares
    \begin{center}
        \begin{tikzcd}
         O \arrow[r]\arrow[d]& A \arrow[d, two heads] \\
         B \arrow[r, tail] & X
        \end{tikzcd} and \begin{tikzcd}
            O \arrow[r]\arrow[d] & B \arrow[d, two heads] \\
            A \arrow[r, tail] & X
        \end{tikzcd}.
    \end{center}
    Then $K_0(\C) = \pi_0 K^\square(\C)$ is given by $\Z\{\textup{Obj}(\C)\}/\sim$
    where $\sim$ is the relation that $[O] = 0 $ and $[A] + [C] = [B] + [C]$ for every distinguished square of the form (\ref{eq:dist_square}). 
\end{prop}

\begin{defn}[{\cite[Definition 2.26]{Calle_Sarazola}}]
    Let $\C$ be a squares category. We call a vertical morphism $f:A\twoheadrightarrow B$ a weak equivalence if \begin{center}
        \begin{tikzcd}
           O \arrow[r] \arrow[d]& A \arrow[d, "f", two heads] \\
           O \arrow[r] & B
        \end{tikzcd}
    \end{center}
    is distinguished.
\end{defn}

Now we turn to categories with covering families and their $K$-theory. These generalize the $K$-theory of assemblers.

We recall that in a category $\C$, an initial object $\emptyset$ is called \textit{strict} if $\hom(X,\emptyset) = \emptyset$ for every $X$ not isomorphic to $\emptyset$. 
\begin{defn}[{\cite[Definition 2.1]{trace_map}}]
    A \textit{category with covering families} consists of a category $\A$ with a strict initial object $\emptyset$, and a collection of finite families of morphisms
    $$\{f_i:A_i \to A \}_{i\in I}$$ called covering families,
    such that the following hold.
    \begin{itemize}
        \item For every object $A$ in $\A$, the singleton family $\{A \xrightarrow{=} A\}$ is covering.
        \item  Given a covering family  $\{f_i:A_i \to A \}_{i\in I}$ and, for all $i\in I$, a covering family $\{G_{ij}:B_{ij} \to A_i \}_{j\in J_i}$, the family 
        $$ \{B_{ij} \xrightarrow{g_{ij}} A_i \xrightarrow{f_i} A \}_{i\in I, j\in J_i}$$ is a covering family.
        \item For any finite set $I$, the family $\{\emptyset \to \emptyset \}_{I}$ is covering.    \end{itemize}
\end{defn}
For $\A$ a category with covering families, we denote by $\A^\circ\subseteq \A$ the full subcategory on objects not isomorphic to $\emptyset$. For $X$ a pointed set, we denote by $X^\circ$ the set of non-base point elements.
\begin{defn}[{\cite[Definition 2.11]{trace_map}}]
    Let $\A$ be a category with covering families, and $X$ a pointed finite set. We define a new category with covering families $X\wedge \A$ with a strict initial object $\emptyset$, and the full subcategory of non-initial objects isomorphic to $\amalg_{x\in X^\circ} \A^\circ$. 
    This category is given the structure of a category with covering families where $\{f_i:A_i \to A \}_{i\in I}$ is covering if all the $f_i$ are in a single copy of $\A$, and form a covering family in $\A$ (the initial object is considered to be in every copy of $\A$).
\end{defn}
This construction defines for every category with covering families $\A$ a functor $(-)\wedge \A$ from finite pointed sets to categories with covering families, since a map of pointed sets $f:X\to Y$ induces a functor 
$$f\wedge \A :X\wedge \A \to Y \wedge \A$$
by re-indexing the copies of $\A$ according to $f$.

\begin{defn}[{\cite[Definition 2.12]{trace_map}}]
    Let $\A$ be a category with covering families. We define the category of covers $\W(\A)$ to be the category with 
   \begin{itemize}
       \item as objects finite tuples $\{A_i\}_{i\in I}$ of objects in $\A^\circ$, including the empty tuple,
       \item a morphism $\{A_i\}_I \to \{B_j\}_J$ is given by a function $\alpha:I\to J$ and a covering family 
       $$\{A_i \to B_j\}_{i\in\alpha^{-1}(j) }$$ for all $j\in J$
   \end{itemize}
\end{defn}
We can see $\W(\A)$ as a pointed category with as basepoint the empty tuple $\emptyset$. This makes the nerve $N_\bullet \W(\A)$ a pointed simplicial set.
\begin{defn}[{\cite[Definition 2.17]{trace_map}}]\label{defn:k_th_cov_cat}
    The $K$-theory spectrum $K(\A)$ of a category with covering families is the symmetric spectrum associated to the $\Gamma$-space
    $$X\mapsto |N_\bullet \W(X\wedge \A)|:\set \to \textup{Spaces}_*.$$
\end{defn}
  This $\Gamma$-space is special and therefore the associated spectrum is a positive $\Omega$-spectrum.

Categories with covering families generalize assemblers.
\begin{defn}
  Let $\A$ be a category with an initial object $\emptyset$. A family of morphisms $\{A_i\to A \}_{i\in I}$ is called a \textit{disjoint family} if for $i\neq j$ in $I$, the square
  \begin{center}
      \begin{tikzcd}
      \emptyset \arrow[r] \arrow[d]& A_i \arrow[d] \\
      A_j \arrow[r] & A
      \end{tikzcd}
  \end{center} is a pullback square. 
\end{defn}
\begin{defn}[{\cite[Definition 2.4]{Z-Kth-ass}}]
    An \textit{assembler} is a small category $\A$ equipped with a Grothendieck topology, that satisfies the following axioms:
    \begin{itemize}
        \item $\A$ has an initial object $\emptyset$ and the empty family covers $\emptyset$,
        \item any two disjoint covering families have a common refinement that is a disjoint covering family,
        \item every morphism in $\A$ is a monomorphism.
    \end{itemize}
\end{defn}
An assembler $\A$ can be given the structure of a category with covering families, where a finite family 
$\{f_i:A_i\to A\}$ is covering if it is a disjoint family, and it is a covering family for the Grothendieck topology (i.e., generates a covering sieve). The $K$-theory of an assembler, as defined in \cite{Z-Kth-ass}, coincides with the $K$-theory of the associated category with covering families. When we use the word assembler, we implicitly mean the associated category with covering families. 

\section{On the $S_\bullet^\square$-construction for squares categories}\label{sect:on_s_bullet}
In this section we deal with the first step in comparing assembler $K$-theory to squares $K$-theory. On the squares category side, we need to show that the squares $K$-theory spectrum can be computed with the $S^\square_\bullet$-construction. In \cite{Calle_Sarazola}, abstract criteria are given for an arbitrary squares category: if it is proto-Waldhausen then the $S^\square_\bullet$-construction computes the same $K$-theory as the $T_\bullet$-construction. Here we formulate a different set of criteria that is easy to verify in geometric examples of interest.

The following construction is used in the proof of \cite[Proposition 2.13]{squares}, and is also defined in detail in \cite[Definition 2.47]{Calle_Sarazola}. It is inspired by the $wT^+_\bullet$-construction in \cite{waldhausen}.
\begin{defn}
    Let $\C$ be a squares category. Then $T^+_n\C$ is the category whose objects are diagrams of the shape
    \begin{equation}\label{eq:obj_T_plus}
        \begin{tikzcd}
            A_0 \arrow[r, tail] & A_1 \arrow[r, tail] & A_2 \arrow[r,tail] & \dots  \arrow[r, tail] & A_n\\
            O \arrow[u, two heads] \arrow[r,tail] & A_{01} \arrow[r, tail] \arrow[u, two heads]& A_{02} \arrow[r, tail] \arrow[u, two heads] & \dots \arrow[r, tail] & A_{0n} \arrow[u, two heads]\\
            & O \arrow[r, tail] \arrow[u, two heads] & A_{12} \arrow[r, tail] \arrow[u, two heads] & \dots \arrow[r, tail] & A_{1n} \arrow[u, two heads]\\
            & & \dots \arrow[u, two heads] & \dots  & \dots \arrow[u, two heads]\\
            & & & O \arrow[r, tail] \arrow[u, two heads]   & A_{n-1,n}. \arrow[u, two heads] \\
            & & & &  O \arrow[u, two heads]
        \end{tikzcd}
    \end{equation} with all squares distinguished squares. Morphisms are given by natural transformations between such diagrams, whose components are vertical maps, such that every naturality square containing vertical and horizontal maps is distinguished, and every naturality square containing only vertical maps is a commuting square in $\E$. 
\end{defn}
We observe that for a morphism in $T^+_n\C$, all components except for those on the top row are weak equivalences, since they form a distinguished square with $O \to O$.

The categories $T^+_n \C$ for $n\geq 0$ assemble into a simplicial category, whose simplicial structure maps are like those of $T_\bullet \C$. In particular, the forgetful functors $T^+_n\C \to T_n \C$ that remember the top row of a diagram, form a simplicial functor.

A similar construction is the following.

\begin{defn}[{\cite[Definition 2.20]{Calle_Sarazola}}]\label{defn:S_square}
    For $\C$ a squares category, we define $S^\square_n \C$ to be the category whose objects are diagrams of the shape 
    \begin{center}
         \begin{tikzcd}
            O \arrow[r,tail] & A_{01} \arrow[r, tail] & A_{02} \arrow[r, tail] & \dots \arrow[r, tail] & A_{0n} \\
            & O \arrow[r, tail] \arrow[u, two heads] & A_{12} \arrow[r, tail] \arrow[u, two heads] & \dots \arrow[r, tail] & A_{1n} \arrow[u, two heads]\\
            & & \dots \arrow[u, two heads] & \dots  & \dots \arrow[u, two heads]\\
            & & & O \arrow[r, tail] \arrow[u, two heads]   & A_{n-1,n}. \arrow[u, two heads] \\
            & & & &  O.\arrow[u, two heads]
        \end{tikzcd}
    \end{center}with all squares distinguished squares. Morphisms are given by natural transformations between such diagrams, whose components are vertical maps, such that every naturality square containing vertical and horizontal maps is distinguished, and every naturality square containing only vertical maps is a commuting square in $\E$.
\end{defn}
All components of a morphism in $S_n^\square \C$ are weak equivalences, since they form a distinguished square with $O\to O$. The categories $S^\square_n \C$ assemble into a simplicial category whose simplicial structure maps are like those in the traditional Waldhausen $S_\bullet$-construction. 

The goal of this section is to show that for squares categories that satisfy some additional axioms, the space $\Omega_O|N_\bullet S^\square_\bullet \C|$ is equivalent to $K^\square(\C) = \Omega_O|N_\bullet T_\bullet \C|$. The following result is standard, and brings us half-way there.
\begin{lem}\label{lem:Tplus_vs_S_square}
    For any squares category $\C$, there is an equivalence $|N_\bullet T_\bullet^+\C|\simeq |N_\bullet S_\bullet^\square\C|$.
\end{lem}
\begin{proof}
    See \cite[Proposition 2.52]{Calle_Sarazola} for a more detailed proof. 
    
    For future reference, we just note that the homotopy equivalence is induced by simplicial functor which we denote in degree $n$ by
    $$G_n:T_n^+ \C \to S^\square_n \C,$$
given by deleting the top row of a diagram. For each $n$, this functor has a homotopy inverse
$$ H_n:S^\square_n\C \to T^+_n\C$$
which repeats the first row. Then $G_n\circ H_n = \textup{id}_{S^\square_n\C}$ and there exists a natural transformation $H_n\circ G_n \to \textup{id}_{T^+_n\C}$.
\end{proof}
Therefore comparing $K$-theory via the $T_\bullet$-construction to $K$-theory via the $S^\square_\bullet$-construction comes down comparing $T^+_\bullet \C$ to $T_\bullet \C$. In other words, the data of a sequence of $n$ composable horizontal morphisms should be equivalent (in a natural way) to the data of a diagram (\ref{eq:obj_T_plus}). It is therefore crucial that any diagram of the shape 
\begin{equation}\label{eq:to_be_completed}
    \begin{tikzcd}
       A \arrow[r, tail] & B \\
       O \arrow[u, two heads] 
    \end{tikzcd}
\end{equation} can be completed to a distinguished square, in a sufficiently natural way. In order to be able to extend this to a functor $T_n\C \to T_n^+\C$, we need to add a number of additional axioms, most of which are intuitive if we think of the object $B\setminus A$ completing (\ref{eq:to_be_completed}) as the set-theoretic complement of $A$ in $B$.

\begin{defn}\label{defn:squares_cat_w_complements}
    A \textit{squares category with complements} is a squares category $\C$ with ambient category $\D$, and a distinguished object $\emptyset$, which satisfies the following.
     \begin{enumerate}[label=(A\arabic*)]
   \item \label{ax:mono} Morphisms in $\mathcal{M}$ are monomorphisms 
         \item \label{ax:compl} For $f:A\to B$ in $\mathcal{M}$, there is a canonical object $B\setminus A$ (the ``complement'')\footnote{It would make sense to denote this object by $B\setminus f(A)$ since it depends on $f$ and not just $A$ and $B$. However, when $f$ is clear from context, we will often denote the complement by $B\setminus A$.} and a morphism $B\setminus A \to B$ in $\E$ such that 
         \begin{center}
             \begin{tikzcd}
                 \emptyset \arrow[r, tail] \arrow[d, two heads]& B\setminus A \arrow[d, two heads]\\
                 A \arrow[r, tail]& B
             \end{tikzcd}
         \end{center}
         is distinguished. The morphism $A\setminus \emptyset \to A$ is an isomorphism for all $A$.
          \item \label{ax:closed_under_pb} Morphisms in $\M$ are stable under pullbacks along morphisms in $\E$, and vice versa.
          \item \label{ax:dist_is_pb} Every distinguished square in $\C$ is a pullback square in $\D$.
        \item \label{ax:compl_functorial} Consider \begin{equation}\label{eq:square_in_prop}
            \begin{tikzcd}
                A\arrow[r, "f", tail] \arrow[d, "h"] & B \arrow[d, "j"]\\
                C \arrow[r, "g", tail] & D
            \end{tikzcd}
        \end{equation} a pullback square with $f,g$ in $\mathcal{M}$, and either $h, j$ in $\M$, or $h, j$ in $\E$. Then there is a unique morphism $B\setminus A \to D \setminus C$ such that the square 
        \begin{center}
            \begin{tikzcd}
                B\setminus A \arrow[r, "j'"] \arrow[d, two heads]& D \setminus C \arrow[d, two heads] \\
                B \arrow[r, "j"] & D
            \end{tikzcd} 
        \end{center} commutes, and is a pullback square. 
        \item \label{ax:compl_excision} For a square of $\M$-morphisms
        \begin{center}
            \begin{tikzcd}
                A_0 \arrow[d, "= ", tail] \arrow[r, tail] & A_1 \arrow[d, tail] \\
                A_0 \arrow[r, tail] & A_2 
            \end{tikzcd}
        \end{center}
       the induced map $(A_2\setminus A_0)\setminus (A_1\setminus A_0) \to A_2\setminus A_1$ 
        is an isomorphism.  
        \item \label{ax:sufficient_for_dist} For a pullback square in $\D$ is of the form
        \begin{center}
            \begin{tikzcd}
                A \arrow[r, "f", tail]\arrow[d, "h", two heads] & B \arrow[d, "j", two heads] \\
                C \arrow[r, "g", tail] & D
            \end{tikzcd}
        \end{center}
        with $f,g$ in $\M$ and $h,j$ in $\E$, if the map $B\setminus A \to D\setminus C$ induced by \ref{ax:compl_functorial} is an isomorphism, then the square is distinguished.
        \item \label{ax:dist_and_excision} If   \begin{center}
            \begin{tikzcd}
                A \arrow[r, tail]\arrow[d, two heads] & B \arrow[d, two heads] \\
                C \arrow[r, tail] & D
            \end{tikzcd}
        \end{center}
        is distinguished, and 
        \begin{center}
            \begin{tikzcd}
                A_0\arrow[d, two heads] \arrow[r, "i", tail]& A\arrow[d, two heads] \\ C_0 \arrow[r, "j", tail] & C
            \end{tikzcd}
        \end{center}
        is a pullback square with $i,j$ in $\M$, then the induced square 
        \begin{center}
            \begin{tikzcd}
                A \setminus A_0 \arrow[r, tail]\arrow[d, two heads] & B  \setminus A_0\arrow[d, two heads] \\
                C \setminus C_0 
                \arrow[r,tail] & D \setminus C_0
            \end{tikzcd}
        \end{center}
        is distinguished. 
    \end{enumerate}
\end{defn}
 Squares categories with complements are designed not only so that the following proposition true, but also to be applicable in geometric cases that are normally studied using assembler $K$-theory. 
\begin{prop}\label{prop:Tplus_vs_T}
    Let $\C$ be a squares category with complements, then the forgetful functor induces an equivalence $|N_\bullet T^+_\bullet \C| \simeq |N_\bullet T_\bullet\C|$.
\end{prop}
\begin{proof}
For each $n$ there is a functor
    $$U_n:T^+_n\C\to T_n\C$$ forgetting everything except the top row,
    and they assemble into a simplicial functor. We show that there is a homotopy inverse $F_n:T_n\C\to T^+_n\C$ in each degree. An object $A_0\rightarrowtail A_1 \rightarrowtail \dots \rightarrowtail A_n$ in $T_n\C$ is sent to the object
    \begin{center}
        \begin{tikzcd}
            A_0 \arrow[r, tail] & A_1 \arrow[r, tail] & A_2 \arrow[r,tail] & \dots  \arrow[r, tail] & A_n\\
            \emptyset \arrow[u, two heads] \arrow[r,tail] & A_1\setminus A_0 \arrow[r, tail] \arrow[u, two heads]& A_2\setminus A_0 \arrow[r, tail] \arrow[u, two heads] & \dots \arrow[r, tail] & A_n\setminus A_0 \arrow[u, two heads]\\
            & \emptyset \arrow[r, tail] \arrow[u, two heads] & A_2\setminus A_1 \arrow[r, tail] \arrow[u, two heads] & \dots \arrow[r, tail] & A_n\setminus A_1 \arrow[u, two heads]\\
            & & \dots \arrow[u, two heads] &   & \dots \arrow[u, two heads]\\
            & & &  &\emptyset \arrow[u, two heads]
        \end{tikzcd}
    \end{center}and a morphism 
\begin{equation}\label{eq:morphism}
\begin{tikzcd}
    A_0 \arrow[r, tail] \arrow[d, two heads] & A_1 \arrow[r, tail]\arrow[d, two heads] & \dots \arrow[r, tail]& A_n \arrow[d, two heads]\\
    B_0\arrow[r, tail] & B_1 \arrow[r, tail] &\dots \arrow[r, tail] & B_n
\end{tikzcd}
\end{equation}
is sent to the morphism with components $A_i\setminus A_j\twoheadrightarrow B_i\setminus B_j$ induced by the distinguished squares
\begin{center}
    \begin{tikzcd}
        A_i\arrow[r, tail]\arrow[d, two heads]& A_j \arrow[d, two heads] \\
        B_i \arrow[r, tail]& B_j 
    \end{tikzcd}
\end{center}
 Note that the components are in $\E$ because of \ref{ax:closed_under_pb}.

    To show that $F_n$ is well-defined on objects, we need to show that for composable $\M$-morphisms $A_i\rightarrowtail A_{i+1}\rightarrowtail A_{i+2}\rightarrowtail A_{i+3}$, the square
\begin{equation}\label{eq:naturality_sqaure}
    \begin{tikzcd}
        A_{i+2}\setminus A_{i+1} \arrow[r, tail] \arrow[d, two heads]& A_{i+3}\setminus A_{i+1} \arrow[d, two heads] \\
        A_{i+2} \setminus A_i \arrow[r, tail] & A_{i+3} \setminus A_i
    \end{tikzcd}
\end{equation}
    is distinguished. We obtain (\ref{eq:naturality_sqaure}) by applying \ref{ax:compl_functorial} to the square 
\begin{equation}\label{eq:hulp_naturality_square}
        \begin{tikzcd}
            A_{i+1}\setminus A_i \arrow[r, tail]\arrow[d, "= ", tail]& A_{i+2}\setminus A_i \arrow[d, tail] \\
            A_{i+1}\setminus A_i \arrow[r, tail] & A_{i+3} \setminus A_{i}.
        \end{tikzcd}
    \end{equation}
    where we implicitly use isomorphisms obtained by \ref{ax:compl_excision}. To see that (\ref{eq:hulp_naturality_square}) commutes, we use that the induced map on complements $A_1\setminus A_0 \to A_3 \setminus A_0$ is unique by \ref{ax:compl_functorial}. The square (\ref{eq:naturality_sqaure}) distinguished by \ref{ax:sufficient_for_dist}, where we use isomorphisms obtained by \ref{ax:compl_excision} again.

To see that $F_n$ is well-defined on morphisms, we need to show that given a morphism (\ref{eq:morphism}), for all $j<i$ the square
\begin{center}
    \begin{tikzcd}
        A_i\setminus A_j \arrow[d, two heads] \arrow[r, tail] & A_{i+1}\setminus A_j \arrow[d, two heads] \\
        B_i \setminus B_j \arrow[r, tail] & B_{i+1}\setminus B_j
        \end{tikzcd}
\end{center}
 is distinguished, and the square
\begin{center}
    \begin{tikzcd}
        A_i\setminus A_j \arrow[r, two heads]\arrow[d, two heads]& B_i\setminus B_j \arrow[d, two heads] \\
        A_i\setminus A_{j-1} \arrow[r, two heads] & B_{i} \setminus B_{j-1}
    \end{tikzcd}
\end{center} 
commutes. The first follows directly from \ref{ax:dist_and_excision}. The second square we obtain from 
\begin{equation}\label{eq:another_hulp_naturality}
    \begin{tikzcd}
        A_j\setminus A_{j-1} \arrow[r, two heads] \arrow[d, tail] & B_j\setminus B_{j-1} \arrow[d, tail] \\
        A_{i}\setminus A_{j-1} \arrow[r, two heads] & B_i \setminus B_{j-1} 
    \end{tikzcd}
\end{equation}
implicitly using isomorphisms obtained from \ref{ax:compl_excision}. That (\ref{eq:another_hulp_naturality}) commutes, follows from \ref{ax:dist_and_excision}.

    It is clear that $U_n\circ F_n = \textup{Id}_{T_n\D}$. On the other hand, let $(A_{ij})$ denote the object
      \begin{center}
        \begin{tikzcd}
            A_0 \arrow[r, tail] & A_1 \arrow[r, tail] & A_2 \arrow[r, tail] & \dots  \arrow[r, tail] & A_n\\
            \emptyset \arrow[u, two heads] \arrow[r, tail] & A_{01} \arrow[r, tail] \arrow[u, two heads]& A_{02} \arrow[r, tail] \arrow[u, two heads] & \dots \arrow[r, tail] & A_{0n} \arrow[u, two heads]\\
            & \emptyset \arrow[r, tail] \arrow[u, two heads] & A_{12} \arrow[r, tail] \arrow[u, two heads] & \dots \arrow[r, tail] & A_{1n} \arrow[u, two heads]\\
            & & \dots \arrow[u, two heads] &   & \dots \arrow[u, two heads]\\
            & & &  &\emptyset \arrow[u, two heads]
        \end{tikzcd}
    \end{center}
  in $T^+_n\C$.  Since the squares
    \begin{center}
        \begin{tikzcd}
            \emptyset \arrow[r, tail] \arrow[d, two heads] & A_{ij} \arrow[d, two heads]\\
            A_{i} \arrow[r, tail] & A_j
        \end{tikzcd}
    \end{center} are distinguished, it follows that there are morphisms $A_{ij}\twoheadrightarrow A_j \setminus A_i$ for all $i,j$. To see that these form a morphism in from $(A_{ij})$ to $F_n\circ U_n((A_{ij}))$, we need to see that each square
    \begin{center}
    \begin{tikzcd}
        A_{ij} \arrow[d, two heads] \arrow[r, tail] & A_{i,j+1} \arrow[d, two heads] \\
        A_j \setminus A_i \arrow[r, tail] & A_{j+1}\setminus A_{i}
        \end{tikzcd}
\end{center}
is distinguished, and each square
\begin{center}
    \begin{tikzcd}
        A_{i+1,j} \arrow[r, two heads]\arrow[d, two heads]& A_j\setminus A_{i+1} \arrow[d, two heads] \\
      A_{i,j} \arrow[r, two heads] & A_j \setminus A_{i}
    \end{tikzcd}
\end{center}
commutes. The first follows directly from \ref{ax:dist_and_excision}. For the second, it suffices to check that each square
\begin{center}
    \begin{tikzcd}
        A_{i+1,j} \arrow[r, two heads]\arrow[d, two heads]& A_j\setminus A_{i+1} \arrow[d, two heads] \\
      A_{j} \arrow[r, two heads] & A_j 
    \end{tikzcd}
\end{center}
commutes, since the morphism $A_j\setminus A_i \twoheadrightarrow A_i$ is a monomorphism. This follows from the fact that 
\begin{center}
    \begin{tikzcd}
        A_{i+1,j} \arrow[r, two heads] \arrow[d, two heads] & A_{i+1,j} \arrow[d, two heads]\\
        A_j \setminus A_{i+1} \arrow[r, two heads] & A_j
    \end{tikzcd}
\end{center}
commutes, which can be deduced from \ref{ax:compl_functorial}. 
\end{proof}
\begin{rmk}
In \cite{Calle_Sarazola} it is shown that the $T_\bullet$-construction and the $S_\bullet^\square$-construction give the same $K$-theory spectrum for a squares category that is {proto-Waldhausen} \cite[Definition 2.38]{Calle_Sarazola}\footnote{Note that according to their conventions, $O$ is terminal in $\E$, and what is $\mathcal{E}$ in our examples, is opposite category $\E^\op$ in their examples.}. Using our convention for the direction of vertical morphisms, the condition for a proto-Waldhausen category is (roughly) that any diagram
\begin{center}
    \begin{tikzcd}
        A \arrow[d, two heads] & \\
        C \arrow[r, tail] & B
    \end{tikzcd}
\end{center}
can be completed to a distinguished square, in a sufficiently natural way. Condition \ref{ax:compl} in the definition of a squares category with complements asks this only for the case where $A=O$, and many other assumptions are added. The conditions of a squares category with complements do not imply that a category is proto-Waldhausen; an example of a squares category with complements that is not proto-Waldhausen, is the squares category of varieties in Definition \ref{defn:squares_cat_var}. On the other hand, a proto-Waldhausen category is not necessarily a squares category with complements. 
\end{rmk}

\section{Comparing the $S_\bullet^\square$-construction to assembler $K$-theory}\label{sect:s_bullet_vs_assembler}
In this section we define the $K$-theory of a squares category $K^\square(\C)$ to the $K$-theory of category with covering families $K(\A)$.

We recall from Definition \ref{defn:k_th_cov_cat} that the $K$-theory of an assembler, or more generally a category of covering families $\A$, is a positive $\Omega$-spectrum whose space in level 1 is given by the geometric realization of the bisimplicial set $N_\bullet\mathcal{W}(S^1_\bullet \wedge \A)$ for $S^1$ the simplicial sphere. Effectively, $S^1_p\wedge \A$ contains $p$ copies disjoint copies of $\A$, and a family $\{A_i\to B \}_{i\in I}$ is covering exactly when the $A_i$'s and $B$ are in the same copy of $\A$ and not isomorphic to $\emptyset$. As a simplicial category, $\mathcal{W}(S^1_\bullet \wedge \A)$ is equivalent to the following.  
\begin{defn}
    We define $\W(\A)^\bullet$ to be the simplicial category given in degree $n$ by the cartesian product
    $$ \W(\A)^n  = \W(\A) \times \dots \times \W(\A).$$
    For an object $(\{A_i\}_{I_1}, \dots, \{A_i \}_{I_n})$ in $\W(\A)^n$, we implicitly assume that the indexing sets $I_1,\dots, I_n$ are disjoint.
    
     The face maps
    $$d_k:\W(\A)^n\to \W(\A)^{n-1}$$
    are as follows: $d_0$ and $d_n$ delete the first and the last object respectively, and $d_k$ adjoins $\{A_{i} \}_{I_k}$ and $\{A_i \}_{I_{k+1}}$ to form $\{A_i\}_{I_k\sqcup I_{k+1}}$, for $k=1, \dots, n-1$.

    The degeneracy maps
    $$s_k:\W(\A)^{n}\to \W(\A)^{n+1}$$ are as follows: $s_k$ inserts the empty tuple before $\{A_i \}_{I_{k+1}}$ for $k=0, \dots, n-1$, and $s_n$ inserts the empty tuple after $\{ A_i\}_{I_n}$. 
\end{defn}
The following is used in e.g. the proof of \cite[Proposition 9.13]{devissage}. We spell it out for the sake of being self-contained.
\begin{lem}
    As simplicial categories, $\mathcal{W}(S^1_\bullet \wedge \A)$ is equivalent to $\W(\A)^\bullet$. 
\end{lem}
\begin{proof}
We denote by $\A_1^\circ,\dots, \A_n^\circ \subseteq S^1_n\wedge \A$ the $n$ subcategories of non-basepoint objects corresponding to the $n$ different copies of $\A$ in $S^1_n\wedge \A$.

    There are equivalences of categories 
    $$F:\W(S^1_n\wedge \A)\to  \W(\A)^n$$
    that are compatible with the face and degeneracy maps. Indeed, for $\{A_i \}_I$ in $\W(S^1_n\wedge \A)$ there is a unique partition $I = I_1\sqcup \dots \sqcup I_n$ such that $A_i$ is in $\A^\circ_k$ if and only if $i\in I_k$ (note that some of the $I_k$ can be empty). We send this to the object 
    $$(\{A_i\}_{I_1},\dots, \{A_i\}_{I_n})$$
    in $ \W(\A)^n$, where some (but never all) of the $\{A_i\}_{I_k}$ can be empty. The empty object of $\W(S^1_n\wedge \A)$ is sent to the tuple of only empty objects in  $ \W(\A)^n$.

    An inverse of $F$ can be given by sending $(\{A_i\}_{I_1},\dots, \{A_i\}_{I_n})$ to $\{A_i \}_{I_1\sqcup \dots \sqcup I_n}$. 
\end{proof}
For $\A$ a category with covering families and $\C$ a symmetric monoidal squares category with complements, both $K(\A)$ and $K^\square(\C)$ are $\Omega$-spectra above level 0. Therefore it is enough to compare the level 1 spaces, in other words, show that $|N_\bullet\mathcal{W}(\A)^\bullet)|$ is homotopy equivalent to $|N_\bullet S^\square_\bullet\C|$. We do this by defining a functor of simplicial categories $\mathcal{W}(\A)^{\bullet}\to S^\square_\bullet \C$. 

In order to be able to do this, we ask that $\A$ and $\C$ share an ambient category $\D$ with a symmetric monoidal structure $\amalg$ that behaves sufficiently like a coproduct (for example, a restricted coproduct). We can then associate to an object $\{A_i \}_I$ in $\W(\A)$, which one may think of as a formal disjoint union of objects in $\A$, an object $\amalg_I A_i$ in $\C$. This allows us to construct from an $n$-tuple of formal disjoint unions an object (\ref{eq:functor_G_n}) in $S_n^\square \C$. Since morphisms in $S_\bullet^\square \C$ are point-wise weak equivalences, we need a map $\amalg_I A_i \to \amalg_{J}B_j$ arising from covers $\{A_i \to B_j \}_{i\in \alpha^{-1}(j)}$ in $\W(\A)$, to be a weak equivalence in $\C$. 

The conditions in the following proposition are such that such a functor $\mathcal{W}(\A)^{\bullet}\to S^\square_\bullet \C$ can be defined, and can be shown to be a levelwise homotopy equivalence using Quillen's Theorem A.

\begin{prop}\label{prop:assmebler_K_th_is_squares_K_th}
Let $\D$ be a category that is the ambient category of a squares category with complements $\C = (\E,\M)$. Assume moreover that there is a category $\mathcal{A}\subseteq \D$ that has the structure of a category with covering families. Moreover, assume the following.
    \begin{enumerate}[label = (B\arabic*)]
        \item \label{ax:coprod} The category $\D$ has a symmetric monoidal structure $\amalg$ for which the distinguished object $\emptyset$ is the unit, and this makes $\C$ a symmetric monoidal squares category.
          \item\label{ax:covering_fam_is_weq} The symmetric monoidal structure $\amalg$ induces a functor
              $$\amalg:\W(\A)\to \E$$sending $\{A_i\}_{i\in I}$ to $\amalg_{i\in I} A_i$, and this functor takes values in weak equivalences.
               \item \label{ax:weq_coprod} Weak equivalences are stable under $\amalg$.


         
            \item\label{ax:coprod_incl_how_ver} For all objects $A, B$, the map $A\to A \amalg B$ induced by $\emptyset \to B$ is both in $\M$ and in $\E$.
             \item \label{ax:compl_of_coprod} For $A\to A\amalg B$ the canonical map, the complement map $(A\amalg B)\setminus A \to A \amalg B$ coincides with the canonical map $B\to A\amalg B$. 
              \item \label{ax:weq_exhibits_dist} For \begin{center}
            \begin{tikzcd}
                A \arrow[r, "f"] \arrow[d, "h"] & B\arrow[d, "j"] \\
                C\arrow[r, "g"] & D
            \end{tikzcd}
        \end{center}
        a pullback square with $f,g$ in $\M$ and $j, h$ in $\E$, if the induced morphism $B\setminus A \to D \setminus C$ is a weak equivalence, then this square is distinguished.
              \item \label{ax:covers_are_monos} All weak equivalences are monomorphisms in $\E$. 
            \item \label{ax:coprod_cover} For $A_1,\dots, A_k$ objects in $\D$, the set $\{A_i\to \amalg_{j=1}^k A_j\}_{i=1}^k$ is a covering family.

        \item \label{ax:complements_are_cover} For $A\to B$ in $\M$, $\{A\to B, B\setminus A \to B \}$ is a covering family.
           \item \label{ax:weq_locally_isom} For every weak equivalence $f:A\to B$ there is a family $\{A_i \to A \}_{i\in I} $ such that $\{A_i \to A \xrightarrow{f} B \}_{i\in I}$ is also a covering family.
        \item \label{ax:covers_refine} Any two covering families
        $\{B_i \to A \}_{I}$ and $\{B'_{i} \to A \}_{I'}$ have a common refinement, in the following sense. There is a covering family $\{C_j \to A \}_J$, and morphisms $\{C_j \}_J\to \{B_i\}_I$ and $\{C_j\}_J\to \{B'_i\}_{I'}$ in $\W(\A)$, such that the resulting diagram
        \begin{center}
            \begin{tikzcd}
                \{C_j\}_J \arrow[d] \arrow[dr] \arrow[r] & \{B_i\}_{I} \arrow[d] \\ \{B'_{i}\}_{I'} \arrow[r] & \{A\}
            \end{tikzcd}
        \end{center} in $\W(\A)$ commutes.
    \end{enumerate}
    Then there is an equivalence of $K$-theory spectra $K(\C)\simeq K^\square(\A)$.
\end{prop}
\begin{proof}
 We compare the simplicial category $\mathcal{W}(\mathcal{A})^\bullet$ to $S^\square_\bullet \mathcal{C}$. To this end, we define a functor of simplicial categories $G_\bullet$ with components
$$G_n:\mathcal{W}(\mathcal{A})^n\to S^\square_n\mathcal{C}$$
	sending an object
	$$(\{A_i \}_{I_1},\dots, \{A_i\}_{I_n})$$
	to the diagram 
     \begin{equation}\label{eq:functor_G_n}
        \begin{tikzcd}
            \emptyset  \arrow[r] &\amalg_{I_1} A_i \arrow[r] & \amalg_{I_2\sqcup I_2}A_i \arrow[r] & \dots \arrow[r] & \amalg_{I_1\sqcup \dots \sqcup I_n} A_i \\
            & \emptyset \arrow[r] \arrow[u] & \amalg_{I_2} A_i \arrow[r] \arrow[u] & \dots \arrow[r] & \amalg_{I_2\sqcup \dots \sqcup I_n}\arrow[u]\\
            & & \dots \arrow[u] &   & \dots \arrow[u]\\
            & & &  &\emptyset \arrow[u]
        \end{tikzcd}
    \end{equation}
    This is well-defined on objects: we can use \ref{ax:weq_exhibits_dist} to check that each square of the form 
    \begin{center}
        \begin{tikzcd}
            \amalg_{I_{k+1}\sqcup \dots \sqcup I_j} A_i\arrow[d]\arrow[r]& \amalg_{I_{k+1}\sqcup \dots \sqcup I_{j+1}} A_i \arrow[d] \\
            \amalg_{I_{k}\sqcup \dots \sqcup I_{j}} A_i \arrow[r] & \amalg_{I_{k} \sqcup \dots \sqcup I_{j+1}} A_i
        \end{tikzcd}
    \end{center}
    is distinguished, since the map on complements is the identity on $\amalg_{I_{j+1}}A_i $ by \ref{ax:compl_of_coprod}. 
To see that $G_n$ is well-defined on morphisms, we need to see that for a morphism $$(\{A_i \}_{I_1},\dots, \{A_i\}_{I_n})\to (\{B_j \}_{J_1},\dots, \{B_j\}_{J_n})$$ in $\W(\A)^n$, 
consisting of maps of finite sets $\alpha_k:I_k\to J_k$ and covering families $\{A_i \to B_j \}_{i\in \alpha_k^{-1}(j)}$ for all $1\leq k \leq n$ and $j\in J_k$,
the squares
\begin{center}
    \begin{tikzcd}
        \amalg_{I_k\sqcup \dots \sqcup I_l}A_i \arrow[d] \arrow[r] & \amalg_{I_k \sqcup \dots \sqcup I_{l+1}} A_i \arrow[d] \\
         \amalg_{J_k\sqcup \dots \sqcup J_l}B_j \arrow[r] & \amalg_{J_k \sqcup \dots \sqcup J_{l+1}} B_j  
    \end{tikzcd}
\end{center} are distinguished. This follows since the map on complements $\amalg_{I_{l+1}}A_i\to \amalg_{J_{l+1}}B_j$ is a weak equivalence, by \ref{ax:covering_fam_is_weq} and \ref{ax:weq_coprod}. \\

We use Quillen's Theorem A to show that each $G_n$ is a homotopy equivalence after realization.  Consider the comma category $G_n/(A_{ij})$ for $(A_{ij})$ an object 
      \begin{center}
        \begin{tikzcd}
            \emptyset  \arrow[r] & A_{01} \arrow[r] & A_{02} \arrow[r] & \dots \arrow[r] & A_{0n}\\
            & \emptyset \arrow[r] \arrow[u] & A_{12} \arrow[r] \arrow[u] & \dots \arrow[r] & A_{1n} \arrow[u]\\
            & & \dots \arrow[u] &   & \dots \arrow[u]\\
            & & &  &\emptyset \arrow[u]
        \end{tikzcd}
    \end{center}
in $S^\square_n\C$. We show that each comma category is contractible, by showing that it is filtered.

We first show that $G_n/(A_{ij})$ is non-empty. For each $j\geq 2$, there are weak equivalences $A_{j-1,j}\to A_{ij}\setminus A_{i,j-1}$ for $0\leq i\leq j-1$. Using \ref{ax:weq_locally_isom} and \ref{ax:covers_refine}, we can find a cover $$\{B_k\to A_{j-1,j} \}_{k\in K_j}$$ such that composing with each of the weak equivalences $A_{j-1,j}\to A_{ij}\setminus A_{i,j-1}$ gives a cover 
$\{B_k\to  A_{ij}\setminus A_{i,j-1} \}_{k\in K_j}$ for all $0\leq i\leq j-1$ simultaneously. We denote by $\{B_k \}_{k\in K_1}$ the one-object set $\{A_{01}\}$, and denote 
$B = (\{B_k \}_{K_1},\dots, \{B_k \}_{K_n})$.

Now we observe that by \ref{ax:complements_are_cover}, each $A_{ij}$ has a cover
$$\{A_{i,i+1}\to A_{ij}, A_{i,i+2}\setminus A_{i, i+1} \to A_{ij},\dots, A_{ij}\setminus A_{i, j-1} \}.$$
Composing with the covers we found before, this means there is a cover
$\{B_k  \to A_{ij}\}_{k\in I_{i+1}\sqcup \dots \sqcup I_{j} }$ for all $0\leq i < j\leq n$. Now \ref{ax:coprod_cover} gives a map $G_n(B)\to (A_{ij})$.  

To see that two objects in $J_n/(A_{ij})$ have a third object lying above it, suppose we have 
$$B = (\{B_k \}_{K_1},\dots, \{B_k \}_{K_n})\ \ \  \textup{and} \ \ \ B' = (\{B'_k \}_{K'_1},\dots, \{B'_k \}_{K'_n})$$
with morphisms $G_n(B)\to (A_{ij})$ and $G_n(B')\to (A_{ij})$. In particular, this means that for $j=1, \dots, n$, there are covers $\{B_k\to A_{j-1,j} \}_{k\in K_j}$ and $\{B'_k \to A_{j-1, j} \}_{k\in K'_j}$. Let $\{B''_k \to A_{j-1,j}\}_{k\in K''_j }$ be a refinement, as in \ref{ax:covers_refine}. Then for
$$B'' = (\{B''_k \}_{K''_1},\dots, \{B''_k \}_{K''_n}),$$
there are maps $B'' \to B $ and $B''\to B'$ which are compatible with the maps $G(B)\to (A_{ij})$ and $G(B')\to (A_{ij})$.

Lastly, we need to show that two parallel morphisms $f,f'$ in $G_n/(A_{ij})$ have a morphism $g$ with $fg=fg'$. For $B$ and $B'$ as above, a morphism from $G(B)\to (A_{ij})$ to $G(B')\to (A_{ij})$ amounts to morphisms $\{B_k \}_{K_l}\to \{B'_{k}\}_{K'_l}$ for $l=1,\dots, n$ such that the induced diagram
\begin{center}
    \begin{tikzcd}
        \amalg_{K_1\sqcup \dots \sqcup K_n} B_k \arrow[rr] \arrow[dr] & & \amalg_{K'_1\sqcup \dots \sqcup K'_n} B'_k  \arrow[dl] \\
        & A_{0n} &
    \end{tikzcd}
\end{center}
commutes, where the maps into $A_{02}$ are weak equivalences. By \ref{ax:covers_are_monos} there is at most one morphism that makes such a diagram commute, so parallel morphisms in $G_n/(A_{ij})$ are always equal.

      This shows that the realizations $|N_\bullet \mathcal{W}(\A)^\bullet|$ and $|N_\bullet T_\bullet \C |$ are homotopy equivalent, as desired.
\end{proof}

\begin{rmk}
     The proposition above gives an alternative proof of \cite[Proposition 9.13]{devissage}. 
\end{rmk}

\section{Examples}\label{sect:examples}
To see that all $K$-theory is squares $K$-theory, we can now take our favorite category with covering families, and look for a squares category with complements that satisfies the conditions  Proposition \ref{prop:assmebler_K_th_is_squares_K_th}. The lists of conditions involved appear long, but are naturally satisfied in several geometric examples. Remarkable is that in all the examples we discuss, there several choices of squares category with complements that give the same $K$-theory spectrum.

\subsection{Assemblers}\label{subsect:assemblers}
For any category with covering families that satisfies a few additional conditions, we can define an associated ``minimal'' squares category with complements that has the same $K$-theory-spectrum.
\begin{prop}\label{prop:artificial_squares_cat_works}
    Suppose $\A$ is a category with covering families that satisfies the following conditions.
    \begin{enumerate}[label=(C\arabic*)]
    \item\label{ax3:coprod} $\A$ is a subcategory of some ambient category $\D$ with coproducts (or restricted pushouts over the initial object $\emptyset$, which we refer to as coproducts for simplicity),
    \item \label{ax3:corpod_cov_fam} for $A_i, i\in I$ a finite set of objects in $\D$, the family of coproduct inclusions 
$$\{A_i \to \amalg_I A_i \}_{i\in I}$$
is a covering family,
\item \label{ax3:disj_fam} For $\{A_i\to A \}_{i
    \in I}$ a covering family, the pairwise pullbacks $A_i\times_A A_j$ are isomorphic to $\emptyset$, 
    \item\label{ax3:cover_mono}  For $\{A_i\to A \}_{i
    \in I}$ a covering family, the induced morphism $\amalg_I A_i \to A$ is a monomorphism.
    \item \label{ax3:refinement} any two covering families of the same object have a common refinement, as in \ref{ax:covers_refine}.
\end{enumerate}
Then there is a squares category with complements $\C^{\min{}}_\D$ such that there is an equivalence of $K$-theory spectra $K(\A)\simeq K^\square(\C^{\min{}}_\D ).$
\end{prop}
We note that conditions \ref{ax3:disj_fam} and \ref{ax3:refinement} are properties that all assemblers have by definition. The other conditions are also natural for most assemblers of geometric objects.

In order for a squares category $\C$ to admit  functors
$$G_n:\W(\A) \to S_n^\square\C$$
as in the proof of Proposition \ref{prop:assmebler_K_th_is_squares_K_th}, it is clear that the horizontal maps need to include all coproduct inclusions, and the weak equivalences need to include maps of the form \begin{equation}\label{eq:weq_from_cover}
    \coprod_I A_i \to A    
    \end{equation} for $\{A_i \to A \}_I$ a covering family, as well as their coproducts. The vertical maps need to include coproduct inclusions as well as weak equivalences, and they need to be stable under pullback by horizontal maps. This motivates the following definition. 
\begin{defn}\label{defn:E}
Let $\A$ be a category with covering families and ambient category $\D$, that satisfy the conditions of Proposition \ref{prop:artificial_squares_cat_works}. We define define $\E\subseteq \D$ to be the wide subcategory whose morphisms are of the form  $$\amalg_I A_i \to \amalg_J B_j$$ induced by a finite sets $\alpha:I\to J$ and maps $A_i\to B_j$, such that for all $j\in J$, $\{A_i\to B_j\}_{i\in \alpha^{-1}(j)}$ is a subset of a covering family of $B_j$.
\end{defn}

   For a coproduct inclusion $\amalg_{I_0} A_i \to \amalg_I A_i$ (where $I_0\subseteq I$), we define its complement to be the coproduct inclusion $\amalg_{I\setminus I_0} A_i \to \amalg_I A_i$.
    We can now define the following squares category with complements. 
\begin{defn}\label{defn:artificial_squares_cat}
   Let $\A$ be a category with covering families and ambient category $\D$, that satisfy the conditions of Proposition \ref{prop:artificial_squares_cat_works}. We define $\C^\min_\D$ to be the squares category with $\M$ the coproduct inclusions, and $\E$ the category defined in Definition \ref{defn:E}.
     Distinguished squares are pullback squares of the form 
    \begin{equation}\label{eq:square_C_min}
        \begin{tikzcd}
            \amalg_{I} A_i \arrow[r] \arrow[d] & \amalg_{J} A_i \arrow[d] \\
            \amalg_{K} B_k \arrow[r] & \amalg_L B_k
         \end{tikzcd}    \end{equation}
         where the induced map $\amalg_{ J\setminus I} A_i \to \amalg_{L \setminus K} B_k$ is a coproduct of maps of the form (\ref{eq:weq_from_cover}).
\end{defn}
\begin{proof}[Proof of Proposition \ref{prop:artificial_squares_cat_works}]
The squares category with complements $\C^\min_\D$ is designed to satisfy the conditions of Proposition \ref{prop:assmebler_K_th_is_squares_K_th}, with respect to $\A$ and the ambient category $\D$.
\end{proof}
\begin{rmk}
    Instead of a squares category, we could have defined a Waldhausen category whose $K$-theory is equivalent to $K(\A)$, as follows. Let $\W ald(\A)$ be the category whose objects are the objects of $\W(\A)$. A morphism $\{A_i\}_I \to \{B_j\}_J$ is given by a \textit{partial} map $\alpha:I\dashrightarrow J$ and maps $A_i \to B_j$ whenever $\alpha$ is defined on $i$, and $\alpha(i)=j$, such that $\{A_i \to B_j\}_{\alpha(i)=j}$ is a subset of a covering family of $B_j$.  We can give $\W ald(\A)$ the structure of a Waldhausen category, where cofibrations are maps 
    $$\{A_i \}_{I_0}  \to \{A_i \}_I$$
    induced by an inclusion $I_0 \subseteq I$, and  weak equivalences are all morphisms that are in $\W(\A)$. For $\{A_i \}_{I_0}  \to \{A_i \}_I$ a cofibration, the cofiber map is then the map $\{A_i \}_I \to \{A_i \}_{I \setminus I_0}$ induced by the partial map $I \dashrightarrow I \setminus I_0$ that is only defined on $I \setminus I_0$. Then $S_\bullet \W ald(\A)$ (where $S_\bullet$ is the classical $S_\bullet$-construction of a Waldhausen category) is equivalent to $S^\square \C_\D^\min$, and therefore homotopy equivalent to $\W(\A)^\bullet$. Better yet, one can show that  $S_\bullet \W ald(\A)$ is  equivalent to $ \W(\A)^\bullet$ for \textit{any} assembler $\A$.  
While this is a fun fact, it seems that there are some advantages to having a more natural-looking squares category $\C$ such that $K(\A) \simeq K^\square(\C)$, see for example Section \ref{sect:application}.
\end{rmk}

In the following sections we will see examples of categories with covering families to which Proposition \ref{prop:artificial_squares_cat_works} applies. We will also consider other squares categories with complements besides $\C^\min_\D$ that have the same $K$-theory spectrum, and compare these to squares categories that have appeared elsewhere in the literature. 

\subsection{Varieties}
The category of varieties (reduced, separated schemes of finite type over a field $k$) fits into this setup. We denote by $\var$ the ambient category of varieties over a fixed field $k$. All arguments work for the category of reduced, finite type schemes over a reduced finite type base scheme as well.

We refer to \cite[Section 2]{voevodsky} for the definition of a Grothendieck topology generated by a cd-structure. 
\begin{defn}\label{defn:cat_cov_var}
We denote by $\A_{\var}$ the category with covering families whose objects are varieties, and whose morphisms are compositions of open and closed immersions. The covering families are the the simple covers in the Grothendieck topology on $\var$ generated by the cd-structure consisting of squares 
\begin{center}
    \begin{tikzcd}
        \emptyset \arrow[r] \arrow[d] & X\setminus U \arrow[d] \\
        U \arrow[r] & X
    \end{tikzcd}
\end{center}
for $U\to X$ an open immersion.
\end{defn}

This is exactly the assembler defined by Zakharevich in \cite[Section 5.1]{Z-Kth-ass}. A squares category of varieties $\C^\min_\var$ can be defined as in Definition \ref{defn:artificial_squares_cat}, and by Proposition \ref{prop:artificial_squares_cat_works}, its $K$-theory spectrum is equivalent to $K(\A_\var)$. Other squares categories of varieties have appeared in the literature, such as \cite[Example 4.4 and Example 4.5]{squares}.
\begin{defn}
    A \textit{piecewise isomorphism} of varieties is a morphism of varieties $f:X\to Y$ such that there is a filtration 
    $$\emptyset = Y_0 \xrightarrow{\circ} Y_1 \xrightarrow{\circ} \dots \xrightarrow{\circ} Y_{n-1}\xrightarrow{\circ} Y_n = Y$$
    of $Y$ by open immersions, such that the restrictions $f|_{f^{-1}(Y_i\setminus Y_{i-1})} : f^{-1}(Y_i \setminus Y_{i-1}) \xrightarrow{\cong} Y_i \setminus Y_{i-1}$ are isomorphisms for $i=1, \dots, n$.
\end{defn}
This definition is equivalent to the more standard definition given in terms of a stratification by closed immersions.
\begin{defn}[{\cite[Example 4.4]{squares}}]\label{defn:squares_cat_var}
    We define a squares category with complements
     $\C_{\var = (\M_\var,\E_\var)}$ with distinguished object $\emptyset$, $\M_\var\subseteq \var$ the category of open immersions and $\E_\var = \var$. The complement of an open immersion $i:U\hookrightarrow X$ is the closed immersion $X\setminus i(U) \to X$. The distinguished squares are pullback squares
    \begin{center}
        \begin{tikzcd}
            A\arrow[d, "h"] \arrow[r, "f", hook] & B \arrow[d, "j"] \\
            C \arrow[r, "g", hook] & D
        \end{tikzcd}
    \end{center} such that the induced map $j:B\setminus f(A) \to D \setminus g(C)$ is a piecewise isomorphism.
\end{defn}
It is straightforward to see that this indeed defines a squares category with complements.

The following lemma about covering families in $\A_\var$ shows that covering families give rise to weak equivalences in $\C_{\var}$ (i.e., piecewise isomorphisms).
\begin{lem}\label{lem:simple_cover_weq}
    For $\{A_i\to A\}_{i\in I}$ a covering family in $\A_{\var}$, there is a filtration of open immersions 
    $$\emptyset= U_0 \to U_1 \to \dots \to U_{n-1} \to U_n= A$$ such that the set of varieties $ \{U_i\setminus U_{i-1}\}_{i=1,\dots, n}$ is equal to the set $\{A_i \}_{i\in I}$. As a consequence, $\amalg_I A_i \to A$ is a piecewise isomorphism.
\end{lem}
\begin{proof}
  Since simple covers (for a topology generated by a cd-structure) are defined inductively, we can prove this by induction. The base case, for the simple cover
  $\{U \to A, A\setminus U \to A \}$, is trivial.

  Now we assume that the statement holds for a simple cover $S=\{A_i \to A \}_{i\in I}$, and we form a new cover $S'$ by composing with $\{U \to A_j, A_j\setminus U \to A_j \}$ for some $j\in I$. The induction hypothesis implies that there is a is a filtration by open immersions 
    $$ \emptyset = U_0 \to U_1 \to \dots \to U_{n-1} \to U_n= A$$ such that each $A_i$ coincides with $U_k\setminus U_{k-1}$ for some $k$. Suppose $A_j = U_k\setminus U_{k-1}$. Then $U \cup U_{k-1}$ is open in $U_k$, and $U_{k-1}$ is open in $U\cup U_{k-1}$, so 
     $$ \emptyset = U_0 \to \dots \to U_{k-1} \to U \cup U_{k-1} \to U_{k} \to \dots  \to U_n= A$$ is a filtration of $A$ that has the desired property with respect to the covering family $S'$. 
\end{proof}

\begin{prop}[{\cite[Proposition 9.13]{devissage}}]
    There is an equivalence of $K$-theory spectra
    $K^\square(\C_{\var})\simeq K(\var)$.
\end{prop}
\begin{proof}
    We need to check that $\C_\var$ and $\A_{\var}$ satisfy the conditions of Proposition \ref{prop:assmebler_K_th_is_squares_K_th}. We note that piecewise isomorphisms are bijections on underlying sets and therefore monomorphisms, so \ref{ax:covers_are_monos} is satisfied. Coproducts open immersions, so \ref{ax:coprod_incl_how_ver} and \ref{ax:coprod_cover} are satisfied. Condition \ref{ax:complements_are_cover} is satisfied by the definition of $\A_{\var}$, and since $\A_\var$ is an assembler, \ref{ax:covers_refine} is also satisfied. 

    Condition \ref{ax:covering_fam_is_weq} follows from Lemma \ref{lem:simple_cover_weq}. Lastly, to see that \ref{ax:weq_locally_isom} holds, suppose $f:A\to B$ is a piecewise isomorphism and
    $$\emptyset = U_0 \to \dots \to U_n = B $$
    is a filtration by open immersions such that 
    $f|_{f^{-1}(U_i\setminus U_{i-1})}:f^{-1}(U_i\setminus U_{i-1}) \to U_i \setminus U_{i-1}$ is an isomorphism for $i=1,\dots, n$. Then $\{f^{-1}(U_i \setminus U_{i-1}) \to A \}_{i=1,\dots, n}$ is a covering family such that composing with $f$ gives a covering family of $B$, as desired.
\end{proof}

\begin{rmk}
There are other squares categories with complements with $\var$ as ambient category, that also satisfy the conditions of Proposition \ref{prop:assmebler_K_th_is_squares_K_th}. For example, in Definition \ref{defn:squares_cat_var} we could have taken the vertical maps to be only injective morphisms of varieties. 

Another squares category with $\var$ as ambient category is defined in \cite[Example 4.4]{squares}, with as horizontal maps the open immersions, and as vertical maps the closed immersions. This squares category fails condition \ref{ax:covering_fam_is_weq} in Proposition \ref{prop:assmebler_K_th_is_squares_K_th}, since the weak equivalences in this squares category are isomorphisms, and therefore do not include all maps of the form $\amalg_I A_i \to A$ associated to a covering family. It is nonetheless true that its $K$-theory spectrum is equivalent to $K(\A_{\var})$, this is shown in the proof of \cite[Theorem 9.1]{devissage}. 
\end{rmk}

\subsection{Definable sets}\label{sect:definable}

In this section we fix an o-minimal structure $(R,<,\dots)$ on a real closed field $R$. We denote by $\defin(R)$ the category whose objects are definable subsets of $R^n$ for varying $n$, and whose morphisms are definable maps, i.e., maps $f:X\to Y$ whose graph is definable.

\begin{lem}\label{lem:coprod}
    The category $\mathrm{Def}(R)$ has coproducts.
\end{lem}
\begin{proof}
Let $X,Y$ be definable subsets of $R^n$ and $R^m$ respectively. Let $a,b$ distinct elements of $R$. Then the definable set 
$$X\amalg Y :=X\times \{0\}^m \times \{a\} \cup \{0\}^n \times Y \times \{b\}$$
admits definable inclusions $X\to X \amalg Y $ and $Y\to X \amalg Y$ and has the universal property of the coproduct of $X$ and $Y$ in $\defin(R)$.
\end{proof}

\begin{defn}
    We define $\mathcal{A}_{\defin(R)}$ to be the wide subcategory of $\defin(R)$ spanned by injective definable maps. We give it the structure of a category with covering families where a family
    $$\{f:A_i\to A \}_{i\in I}$$
    is covering if and only if $A = \bigcup_I A_i$ and the pairwise pullbacks $A_i\times_A A_j$ are empty.
\end{defn}
This category with covering families is in fact an assembler, see also \cite[Example 3.4]{Z-Kth-ass}.
We can define a squares category of definable sets $\C^\min_{\defin(R)}$ as in \ref{defn:artificial_squares_cat}, and by Proposition \ref{prop:artificial_squares_cat_works} it has the same $K$-theory spectrum as the category with covering families $\A_{\defin(R)}$. Another possible choice of squares category is the following.
\begin{defn}\label{defn:square_def_sets} 
    We define a squares category with complements  $\C_{\defin(R)}$ with $\defin(R)$ as ambient category, with both $\M_{\defin(R)}$ and $\E_{\defin(R)}$ the definable injections. The complement of a definable injection $f:A\to B$ is the set-theoretic complement $B\setminus f(A) \to B$. A square
    \begin{center}
        \begin{tikzcd}
            A \arrow[d, "h"] \arrow[r, "f"] & B\arrow[d, "j"] \\
            C \arrow[r, "g"] & D
        \end{tikzcd}
    \end{center}
    is distinguished if and only if it is a pullback square and the induced morphism $B\setminus f(A) \to D\setminus g(C)$ is a definable bijection. 
 \end{defn} 
 It is, again, easy to verify that this defines a squares category with complements.

\begin{prop}\label{prop:eq_K_spec_definable}
    There is an equivalence of $K$-theory spectra
    $$K(\A_{\defin(R)})\cong K(\C_{\defin(R)}).$$
\end{prop}
\begin{proof}
The conditions of Proposition \ref{prop:assmebler_K_th_is_squares_K_th} are easily verified in this example. Note that the weak equivalences are exactly the definable bijections, in other words, the isomorphisms in $\defin(R)$.
\end{proof}\begin{rmk}

In the proposition above, the structure $R$ does in fact not need to be o-minimal; we assumed it for simplicity, and but it becomes relevant only later in this section, and in Section \ref{sect:application}. \end{rmk}

 In \cite[Example 1.13, Example 1.24]{Ming}, Ng defines a pCGW-category of definable sets, with the same horizontal and vertical morphisms and the same distinguished squares as the category $\C_{\defin(R)}$ in Definiton \ref{defn:square_def_sets}. The $K$-theory spectrum of a pCGW-category can be computed using an $S_\bullet$-construction similar to Definition \ref{defn:S_square};  the $S_\bullet$-construction for pCGW-categories, \cite[Construction 1.26]{Ming}, is a simplicial set rather than a simplicial category. The elements in degree $n$ are exactly the objects of $S^\square_n\C_{\defin(R)}$. However, all morphisms in $S^\square_\bullet\C_{\defin(R)}$ are invertible. This implies that the simplicial set and the simplicial category have the same geometric realization. The argument for this is sketched in \cite[Lemma 1.4.1]{waldhausen}. Since this reasoning will be used more often, we fill in the details in the following lemma. 
 \begin{lem}[{\cite[Corollary to Lemma 1.4.1]{waldhausen}}]\label{lem:simplicial_set_vs_cat}
     Let $\mathfrak{X}_\bullet$ be a simplicial category, and $X_\bullet$ the simplicial set given by $X_n = \textup{Obj}(\mathfrak{X_n})$. Suppose that each category $\mathfrak{X}_n$ is a groupoid.
     Then there is weak homotopy equivalence of spaces
     $$|X_\bullet| \simeq |N_\bullet(\mathfrak{X}_\bullet)|$$ where $|N_\bullet(\mathfrak{X}_\bullet)|$ is the geometric realization of the bisimplicial set $$n,k \mapsto N_n(\mathfrak{X}_k).$$
 \end{lem}
 \begin{proof}
     One way to realize the bisimplicial set $N_\bullet(\mathfrak{X}_\bullet)$, is to apply geometric realization level-wise, and then take the geometric realization of the resulting simplicial space
     $$n\mapsto |N_n(\mathfrak{X}_\bullet)|.$$
     This simplicial space is a good simplicial space, since the degeneracy maps exhibit $|N_n(\mathfrak{X}_\bullet)|$ as closed subspace and retract of $|N_{n+1}(\mathfrak{X}_\bullet)|$. Therefore the realization of this simplicial space is weakly equivalent to its fat realization (\cite[{Proposition A.1 (iv)}]{segal}).

 We note that it zeroth space of $n \mapsto |N_n(\mathfrak{X}_\bullet)|$ is $|N_0(\mathfrak{X}_\bullet)| = |X_\bullet|$, and the face and degeneracy maps are all homotopy equivalences. Moreover, each $|N_n(\mathfrak{X}_\bullet)|$ is a cofibrant space. Therefore by \cite[Theorem 4.3 and Remark 4.4]{Wang}, the fat realization $$\| n \mapsto |N_n(\mathfrak{X}_\bullet) | \|$$ is weakly equivalent to the homotopy colimit of the diagram. Therefore it is weakly equivalent to any of its spaces, and in particular to $|X_\bullet|$, as desired.
 \end{proof}
\begin{cor}\label{cor:K_square_def_vs_K_pCGW_def}
The $K$-theory spectrum $K^\square(\C_{\defin(R)})$ is equivalent to the $K$-theory spectrum of $\defin(R)$ as pCGW-category.
\end{cor}

Let $\A^{\leq n}_{\defin(R)} \subseteq \A_{\defin(R)}$ the subcategory with covering families spanned by the definable sets of dimension at most $n$.  The natural inclusions $\A^{\leq n}_{\defin(R)} \hookrightarrow \A^{\leq n+1}_{\defin(R)} $  induce a filtration
\begin{equation}\label{eq:filtration_K(Def)}
  * = K(\A_{\defin(R)}^{\leq-1})\to K(\A^{\leq 0}_{\defin(R)})\to K(\A^{\leq 1}_{\defin(R)})\to \dots \to K(\A_{\defin(R)}).  
\end{equation}
Let $\defin(R)^n \subseteq \defin(R)$ denote the full subcategory on definable set of dimension exactly $n$. 
\begin{defn}
    We denote by $\A_{\defin(R)^n}$ the wide subcategory of $\defin(R)^n$ spanned by definable injections. We give it the structure of a category with covering families where 
    $$\{f_i:A_i\to A\}_{i\in I}$$
    is a covering family if there are morphisms $g_j:B_j\to A$ for $j\in J$, with $B_j$ of dimension $\leq n-1$, such that 
    $$\{f_i:A_i\to A\}_{i\in I} \cup \{g_j:B_j\to A \}_{j\in J}$$
    is a covering family in $\A_{\defin}$.
\end{defn}
We note that this category with covering families is exactly the assembler $\A^{\leq n}_{\defin(R)}\setminus \A^{\leq n-1}_{\defin(R)}$ as defined in \cite[Definition 2.9]{Z-Kth-ass}. Therefore by \cite[Theorem D]{Z-Kth-ass}, the associated graded of the filtered spectrum (\ref{eq:filtration_K(Def)}) is  $\bigvee_{n\geq 0}K(\A_{\defin(R)^n})$. It is easy to see that the assembler $\A_{\defin(R)^n}$ satisfies the conditions of Proposition \ref{prop:artificial_squares_cat_works}, and therefore its $K$-theory spectrum is equivalent to the $K$-theory spectrum of a squares category.

\subsection{Polytopes}
In this section we fix a euclidean or hyperbolic $n$-dimensional geometry $X$ for $n\ge 1$, and a subgroup  $G$ of the isometry group of $X$. We assume that $G$ is such that for any two $n$-polytopes $P, G$, there is a $g\in G$ such that $gP$ and $Q$ are disjoint. This holds for example if $X$ is a euclidean geometry and $G$ contains all translations.

 We denote by $\P^X_G$ the category of covering families defined in \cite[Example 2.9]{trace_map}. Objects are finite unions of $n$-dimensional closed simplices, and a morphisms $g:P\to Q$ is given by an isometry $g\in G$ such that $gP\subseteq Q$. A covering family is a collection of maps $$\{g_i:P_i\to P \}$$ such that the images $g_i P_i$ have pairwise intersections of measure zero, and cover $P$. 
This category with covering families is in fact an assembler, see also \cite[Example 3.6 and Section 5.2]{Z-Kth-ass}.

We define the following ambient category which admits $\P^X_G$ as a subcategory, and has better categorical properties.
\begin{defn}\label{defn:ambient_polytope_cat}
    We define $\D^G_X$ to have the same objects as $\P^X_G$. A morphism $g:P\to G$ in $\P^X_G$ is given by a decomposition $P = P_1\cup \dots \cup P_k$ such that the inclusions
    $$\{P_i \to P\}_{i=1,\dots, k}$$ form a covering family in $\P^G_X$, and by group elements $g_i\in G$ such that $g_iP_i \subseteq Q$. Two morphisms $P\to G$ are considered the same if the defining decompositions have a common refinement, on which they are given by the same elements in $G$.
\end{defn}
We note that a morphism $g:P\to Q$ in $\D^G_X$, as in the definition above, is invertible exactly when $$\{g_i:P_i \to Q\}_{i=1,\dots, k}$$ is a covering family, in other words, the images cover $Q$ and are pairwise almost disjoint. In that case, the inverse is given by the decomposition of $Q$ in $Q=g_1(P_1) \cup \dots \cup g_k(P_k)$ and the isometries $g_1^{-1}, \dots, g_k^{-1}$.  

We call a morphism $P \to Q$ in $\D^G_X$ a \textit{polytope inclusion} if it is a morphism in $\P^G_X$, in other words, it consists of a single $g\in G$ with $gP \subseteq Q$. 
\begin{prop}
The category $\D^G_X$ has coproducts.
\end{prop}
\begin{proof}
    Let $P, Q$ be polytopes in $X$. By assumption, there is an isometry $g\in G$ such that $gP$ and $Q$ are disjoint. Then the union $gP \cup G$ has the universal property of the coproduct. Indeed, for $h_1:P \to R$ and $h_2:Q \to R $, a morphism $h: gP \cup Q \to R$ can be given by defining $h$ to be $h_1 g^{-1}$ on $gP$ and $h_2$ on $Q$. 
\end{proof}

The ambient category $\D^G_X$ and the assembler $\P^G_X$ satisfy the assumptions of Proposition \ref{prop:artificial_squares_cat_works}, so there is a squares category $\C^\min_{\D^X_G}$ such that $K(\P^G_X) \simeq K^\square(\C^\min_{\D^G_X})$. Another squares category can be defined as follows.
\begin{defn}\label{def:squares_cat_polytopes}
    We define a squares category with complements $\C_X^G$ which has $\D_X^G$ as ambient category, with the empty polytope $\emptyset$ as distinguished object,  $\M$ the category of polytope inclusions, and $\E=\D^G_X$. The complement of a polytope inclusion $g:P \to Q$ is the inclusion of the closure $\overline{Q\setminus gP} \to Q$. The distinguished squares are the pullback squares
    \begin{equation*}
        \begin{tikzcd}
            P \arrow[d] \arrow[r,  "g"] & Q \arrow[d]\\
            R \arrow[r, "h"] & S 
        \end{tikzcd}
    \end{equation*}
    such that the induced morphism on the complements $\overline {Q\setminus gP} \to  \overline{S\setminus hR}$ is invertible.
\end{defn}
\begin{prop}\label{prop:polytopes_squares_cov_agree}
    There is an equivalence of $K$-theory spectra
    $$K^\square(\C^G_X) \simeq K(\P^G_X).$$
\end{prop}
\begin{proof}
    It is straightforward to verify the conditions on Proposition \ref{prop:Tplus_vs_T}.  The conditions of Proposition \ref{prop:assmebler_K_th_is_squares_K_th} are also easily verified; in particular \ref{ax:covering_fam_is_weq} and \ref{ax:weq_locally_isom} follow from the way that morphisms in $\D^G_X$ are defined.
\end{proof}

    In \cite[Example 2.14]{Calle_Sarazola} a proto-Waldhausen category $\mathbb{P}^n_G$ of euclidean polytopes in $\R^n$ is defined, which is in fact a squares category with complements. It has the same horizontal maps and complements as $\C^G_{E_n}$. The distinguished squares are similar to those in $\C^{G}_{E^n}$, except that the vertical maps are required to be polytope inclusions, and therefore the weak equivalences are maps $P\to Q$ that are given by an isometry $g$ such that $gP = Q$.
    
    \begin{prop}\label{prop:squares_cat_CS_agrees}
        There is an equivalence of $K$-theory spectra $K(\P^G_{E_n})\simeq K^\square(\mathbb{P}^n_G)$.
    \end{prop}
    \begin{proof}
    There is an inclusion of simplicial categories 
    $i:S_\bullet^\square \mathbb{P}^n_G \hookrightarrow S_\bullet^\square \C^G_{E^n}$. We note that every object in $S_\bullet^\square \C^G_{E^n}$ is isomorphic to an object in $\im(i)$. Indeed, let $A$ denote an object in $S_k^\square \C^G_{E^n}$, and $H_k(A)$ its image in $T^+_k\C^G_{E_n}$, as in the proof of Lemma \ref{lem:Tplus_vs_S_square}. The proof of Proposition \ref{prop:Tplus_vs_T} gives an object $F_kU_kH_k(A)$ in $T^+_k\C^G_{E_n}$ with a morphism $H_k(A) \to F_kU_kH_k(A)$, and applying $G_k$ gives a morphism 
    $A \to G_kF_kU_kH_k(A)$ in $S_k^\square \C^G_{E^n}$, which is an isomorphism since all weak equivalences in $\C^G_{E^n}$ are invertible. Moreover, $G_kF_kU_kH_k(A)$ is in $\im(i)$.
    
    Now for $k\geq 0,$ let $\langle \im(i)\rangle_k \subseteq S_k^\square \C^G_{E^n}$ denote the full subcategory spanned by objects in the image of $i$. These assemble into a simplicial category $\langle \im(i)\rangle_\bullet$ and by the above, the inclusion into $S_\bullet^\square \C^G_{E^n}$ induces a homotopy equivalence 
    $$|\langle \im(i)\rangle_\bullet | \simeq |S^\square_\bullet \C^G_{E_n}|.$$

On the other hand, $\langle \im(i)\rangle_\bullet$ and $S_\bullet^\square \mathbb{P}^n_G $ have the same objects, and in both categories, all morphisms are invertible. Therefore in both cases, by Lemma \ref{lem:simplicial_set_vs_cat}, the geometric realization is homotopy equivalent to the geometric realization of the simplicial set $\textup{Obj}(S_\bullet^\square \mathbb{P}^n_G)$. Combined, this shows that
 $$ |S^\square_\bullet \mathbb{P}^n_G| \simeq  |\langle \im(i)\rangle_\bullet | \simeq |S^\square_\bullet \C^G_{E_n}|.$$
In particular, by Proposition \ref{prop:polytopes_squares_cov_agree}, for both $\mathbb{P}^n_G$ and $\C^G_{E_n}$ the squares $K$-theory spectrum is equivalent to $K(\P^G_{E_n})$. 
    \end{proof}

\begin{rmk}
    In this section we assumed the geometry to be euclidean or hyperbolic, since this is needed for coproducts to exist in $\D^G_X$. For $X$ a spherical geometry, coproducts do not necessarily exist. One can still show that even in this case, $K(\P^G_X)$ is the $K$-theory of a squares category. This can be done by considering a category of ``spherical polytopes and formal coproducts''. Its subcategory generated by polytope inclusions and coproduct inclusions can be given the structure of a category with covering families, and dévissage shows that its algebraic $K$-theory is equivalent to $K(\P^G_X)$. Moreover, for this category with covering families, and its ambient category of spherical polytopes and formal coproducts, the conditions of Proposition \ref{prop:artificial_squares_cat_works} are fulfilled.
\end{rmk}

\section{Application: a derived Euler characteristic}\label{sect:application}

We return to the example in Section \ref{sect:definable} of definable sets in an o-minimal structure $\defin(R)$. The order topology on $R$ induces a topology on each $R^n$ and therefore on every definable set.

\begin{defn}
    Let $\defin(R)^{\lc}\subseteq \defin(R)$ be the full subcategory on definable sets that are locally closed (as subset of $R^n$ for some $n$).
\end{defn}
\begin{defn}
    We define $\A_{\defin(R)}^{\lc} $ to be the wide subcategory of $\defin(R)^\lc$ spanned by injective definable maps. We give it the structure of a category with covering families where a family
    $$\{f:A_i\to A \}_{i\in I}$$
    is covering if and only if it is a covering family in $\A_{\defin(R)}$. 
\end{defn}
\begin{lem}\label{lem:eq_K_th_def_vs_def_lc}
    There is an equivalence of $K$-theory spectra 
    $$K(\A_{\defin(R)})\simeq K(\A_{\defin(R)}^{\lc}).$$
\end{lem}
\begin{proof}
    We note that $\A_{\defin(R)}^{\lc}$ is a full sub-assembler of $\A^{\defin(R)}$.  Any definable set $X\subseteq R^n$ has a cell-decomposition $X= \bigcup_I C_i$ and every cell is locally closed (\cite[page 51]{Dries}). The set $\{C_i \to X\}_{i\in I}$ is a covering family in $\A_{\defin(R)}$. This shows that every object in $\A_{\defin(R)}$ is covered by objects in $\A_{\defin(R)^{\lc}}$ and therefore by \cite[Theorem B]{Z-Kth-ass}, the inclusion induces an equivalence of $K$-theory spectra $K(\A_{\defin(R)}^{\lc}) \simeq K(\A_{\defin(R)}).$
    \end{proof}

Just like in Section \ref{sect:definable}, we can define a squares category $\C^\min_{\defin(R)^{\lc}}$ such that $K(\A_{\defin(R)})\simeq K(\A_{\defin(R)}^{\lc}) \simeq K^\square(\C^\min_{\defin(R)^\lc}).$ However, for the purpose of defining the derived Euler characteristic, we are interested in another squares category of locally closed definable sets.

We recall that definable bijections are always invertible, but they do not need to be continuous, and neither is their inverse in general.
\begin{defn}
  We call a definable bijection a \textit{definable homeomorphism} if it is continuous, and its inverse is continuous.  
\end{defn}
We note that the definition above is in contrast to some references that call a definable map a definable homeomorphism as soon as it has a definable inverse (where neither the map nor its inverse needs to be continuous).

  \begin{defn}
    A \textit{definable open embedding} is a definable injection $f:A\to B$ that is a definable homeomorphism onto $f(A)$, where $f(A)$ is open in $B$.

    A \textit{definable closed embedding} is a definable injection $f:A\to B$ that is a definable homeomorphism onto $f(A)$, where $f(A)$ is closed in $B$. 
\end{defn}
We note that definable open embeddings and definable closed embeddings are in particular (topological) open and closed embeddings, respectively. 

For $f:A\to B$ any definable injection, we define its complement map to be the inclusion $B\setminus f(A) \to B$. Clearly, the complement of a definable open embedding is a definable closed embedding, and vice versa.
\begin{defn}
    We define $\C_{\defin(R)^\lc}$ to be the squares category with complements with as distinguished object the empty set $\emptyset$, $\M$ the definable closed embeddings and $\E = \defin(R)$. The distinguished squares are the pullback squares \begin{center}
          \begin{tikzcd}
            A\arrow[d, "h"] \arrow[r, "f"] & B \arrow[d, "j"] \\
            C \arrow[r, "g"] & D
        \end{tikzcd}
    \end{center}
    such that the induced morphism $B\setminus f(A) \to D \setminus g(C)$ is a definable bijection.
\end{defn}
By the exact same proof as Proposition \ref{prop:eq_K_spec_definable}, combined with the equivalence shown in Lemma \ref{lem:eq_K_th_def_vs_def_lc}, we obtain the following.
\begin{prop}\label{prop:A_def_vs_def_lc}
 There is an equivalence of $K$-theory spectra
    $$K(\A_{\defin(R)})\cong K(\C_{\defin(R)^\lc}).$$
\end{prop}
Now we introduce yet another squares category of locally closed definable sets.
\begin{defn}
    We define $\widetilde{\C}_{\defin(R)^\lc}$ to be the squares category with complements with distinguished object $\emptyset$, $\M$ the definable closed embeddings and $\E$ the definable open embeddings. The distinguished squares are the pullback squares \begin{center}
          \begin{tikzcd}
            A\arrow[d, "h"] \arrow[r, "f"] & B \arrow[d, "j"] \\
            C \arrow[r, "g"] & D
        \end{tikzcd}
    \end{center}
  such that $j(B)\cup g(C) = D$ (or equivalently, such that the induced map $B\setminus f(A) \to D \setminus g(C)$ is a definable homeomorphism). 
\end{defn}
Proposition \ref{prop:assmebler_K_th_is_squares_K_th} does not apply here, since the maps $\amalg_I A_i \to A$ coming from covers are generally not definable homeomorphisms and are therefore not weak equivalences in $\widetilde{\C}_{\defin(R)^\lc}$. However, the squares category $\widetilde{\C}_{\defin(R)^\lc}$ is still relevant by the following proposition, comparable to Corollary \ref{cor:K_square_def_vs_K_pCGW_def} and Proposition \ref{prop:squares_cat_CS_agrees}. 
\begin{prop}\label{prop:K_cov_def_vs_def_lc}
    There is an equivalence of $K$-theory spectra  $$K(\A_{\defin(R)})\cong K(\widetilde\C_{\defin(R)^\lc}).$$
\end{prop}
\begin{proof}
The proof is essentially the same as the proof of Proposition \ref{prop:squares_cat_CS_agrees}: there is an essentially surjective inclusion $i:S^\square_\bullet \widetilde \C_{\defin(R)^\lc} \hookrightarrow S^\square_\bullet\C_{\defin(R)^\lc}$ whose image spans a full simplicial sub-category $\langle \im(i)\rangle_\bullet \subseteq S^\square_\bullet\C_{\defin(R)^\lc},$   
such that $$|S^\square_\bullet  \widetilde \C_{\defin(R)^\lc} | \simeq | \langle \im(i)\rangle_\bullet| \simeq | S^\square_\bullet\C_{\defin(R)^\lc}| $$
where the first equivalence comes from Lemma \ref{lem:simplicial_set_vs_cat}.     
    Combined with Proposition \ref{prop:A_def_vs_def_lc}, this gives the desired equivalence.
\end{proof}

To lift the definable Euler characteristic, defined in the introduction, to $K$-theory spectra, we use a homology theory of \textit{ $\Z/2$-valued constructible functions}. For $X$ a definable set, a function $f:X\to \Z/2$ is constructible if for all $n\in \Z/2$, the set $f^{-1}(n)$ is definable. For $f$ a $\Z/2$-valued constructible function, we define $\supp(f) = f^{-1}(1)$. Note that any constructible function $f:X\to \mathbb{Z}/2$ can (non-uniquely) be written as a finite sum $f= \sum_{i\in I} 1_{C_i}$ of indicator functions of cells

We denote by $CF_{\leq i}(X)$ the $\Z/2$-module of constructible functions such that $\supp (f)$ is of dimension $\leq i$. We define a chain complex $CF_*(X)$ where $$CF_i(X)  = CF_{\leq i}(X)/CF_{\leq i-1}(X).$$ The differential is given by $\partial(1_C)=1_{\partial C}$ whenever $C$ is a cell, extended linearly. We define $CH_*(X)$ to be the homology of $CF(X)$. This invariant is of interest to us for the following reason.
\begin{prop}[{\cite[Corollary 7.11]{euler_calc}}]
    For $X$ a locally closed definable set, we have $$\chi^\defin(X) = \sum_{i=0}^{\dim(X)} (-1)^{i} \dim CH^i(X).$$
\end{prop}

The chain complex $CF_*(X)$, and therefore its homology, is contravariantly functorial in open embeddings: for $i:U\to X$ an open embedding, we define $i^*(1_C) = 1_{i^{-1}(C)}$ on indicator functions of cells. On the other hand, $CF_*(X)$ is covariantly functorial in closed embeddings: for $j:X\to Y$ a closed embedding, we define $j_!(1_C) = 1_{j(C)}$. Both of these functors commute with the differential, and preserve dimension. 

As in \cite[Section 5.1]{manifolds}, we denote by $\ch^\hb_{\Z/2}$ the Waldhausen category of homologically bounded chain complexes over $\Z/2$. By abuse of notation, we denote the associated squares category, as in Example \ref{ex:wald_square}, by $\ch^\hb_{\Z/2}$ as well.
\begin{prop}
    The assignment $X\mapsto CF_*(X)$ extends to a simplicial functor
    $$CF_*:S_\bullet^\square\widetilde \C_{\defin(R)^\lc} \to S^\square_\bullet\ch^\hb_{\Z/2}$$
\end{prop}
\begin{proof}
For $j$ a closed embedding, $j_!$ is level-wise injective and therefore a cofibration. Hence, to see that $CF_*$ is well-defined on objects, it suffices to see that for a square of the form 
\begin{center}
    \begin{tikzcd}
        \emptyset \arrow[r] \arrow[d] & X\setminus C \arrow[d, "i"]\\
        C \arrow[r, "j"] & X
    \end{tikzcd}
\end{center}
the square
\begin{center}
    \begin{tikzcd}
        CF_*(C) \arrow[r, "j_!"] \arrow[d] & CF_*(X) \arrow[d, "i^*"]\\
        0 \arrow[r] & CF_*(X\setminus C)
    \end{tikzcd}
\end{center}
is a pushout, which is easy to check. The morphisms in $S_\bullet^\square\widetilde \C_{\defin(R)^\lc}$ are level-wise definable homeomorphisms, and therefore sent to level-wise isomorphisms in $S^\square_\bullet\ch^\hb_{\Z/2}$. The functor $CF_*$ commutes with simplicial structure maps, and is therefore a simplicial functor, as desired. 
\end{proof}
Combined with Proposition \ref{prop:K_cov_def_vs_def_lc} we obtain the following corollary.
\begin{cor}\label{cor:map_of_spectra_def_ch}
    There is a map of $K$-theory spectra 
    $$K(\A_{\defin(R)}) \to K(\ch^\hb_{\Z/2}).$$
\end{cor}
There is a standard equivalence $K(\ch^\hb_{\Z/2}) \simeq K(\Z/2)$, and we recall that $K_0(\Z/2)= \Z$. 
\begin{prop}\label{prop:map_lifts_eulerchar}
    On $\pi_0$, the map  $$K(\A_{\defin(R)}) \to K(\Z/2)$$ agrees with the definable Euler characteristic.
\end{prop}
\begin{proof}
    The proof goes as the proof of \cite[Theorem 5.1]{manifolds}.
\end{proof}
\begin{rmk}
    We used $\Z/2$-valued constructible functions for convenience, but we expect that with more care, a map  $K(\A_{\defin(R)}) \to K(\Z)$ can be constructed which lifts the Euler characteristic, using $\Z$-valued constructible functions. Another avenue to obtain this, could be to use Borel-Moore homology or compactly supported cohomology with $\Z$-coefficients; see also \cite[Remark 7.12]{euler_calc}.
\end{rmk}

\bibliographystyle{alpha}
\bibliography{bibliography}

\end{document}